\documentclass{siamart250211}
\newsiamthm{assumption}{Assumption}
\newsiamremark{remark}{Remark}
\usepackage{lipsum}
\usepackage{amsfonts}
\usepackage{graphicx}
\usepackage{epstopdf}
\usepackage{algorithmic}
\usepackage{mathtools}
\allowdisplaybreaks

\newcommand{\rd}{\, \mathrm{d}}

\newcommand{\wor}{\mathrm{wor}}

\newcommand{\bszero}{\boldsymbol{0}}
\newcommand{\bsh}{\boldsymbol{h}}

\newcommand{\bsk}{\boldsymbol{k}}
\newcommand{\bsl}{\boldsymbol{\ell}}
\newcommand{\bsq}{\boldsymbol{q}}
\newcommand{\bsx}{\boldsymbol{x}}

\newcommand{\bsy}{\boldsymbol{y}}
\newcommand{\bsz}{\boldsymbol{z}}
\newcommand{\bsgamma}{\boldsymbol{\gamma}}

\newcommand{\FF}{\mathbb{F}}
\newcommand{\II}{\mathbb{I}}
\newcommand{\NN}{\mathbb{N}}
\newcommand{\PP}{\mathbb{P}}
\newcommand{\RR}{\mathbb{R}}
\newcommand{\ZZ}{\mathbb{Z}}
\newcommand{\Acal}{\mathcal{A}}

\newcommand{\wal}{\mathrm{wal}}

\DeclareMathOperator{\tr}{tr}

\title{Quasi-Monte Carlo hyperinterpolation}

\author{
Congpei An\thanks{School of Mathematics and Statistics, Guizhou University, Guiyang 550025, Guizhou, China 
  (\url{ancp@gzu.edu.cn}, \url{andbachcp@gmail.com})}
\and
Mou Cai\thanks{Graduate School of Engineering, The University of Tokyo, 7-3-1 Hongo, Bunkyo-ku, Tokyo 113-8656, Japan 
  (\url{caimoumou@g.ecc.u-tokyo.ac.jp}, \url{goda@frcer.t.u-tokyo.ac.jp})}
\and
Takashi Goda\footnotemark[2]
}

\begin{document}
\maketitle
\begin{abstract}
This paper studies a generalization of hyperinterpolation over the high-dimensional unit cube. Hyperinterpolation of degree \( m \) serves as a discrete approximation of the \( L_2 \)-orthogonal projection of the same degree, using Fourier coefficients evaluated by a positive-weight quadrature rule that exactly integrates all polynomials of degree up to \( 2m \). Traditional hyperinterpolation methods often depend on exact quadrature assumptions, which can be impractical in high-dimensional contexts.
We address the challenges and advancements in hyperinterpolation, bypassing the assumption of exactness for quadrature rules by replacing it with quasi-Monte Carlo (QMC) rules and propose a novel approximation scheme with an index set \( I \), which is referred to as QMC hyperinterpolation of range \( I \). In particular, we provide concrete construction algorithms for QMC hyperinterpolation with certain lattice rules. Consequently, we show that QMC hyperinterpolation achieves accuracy comparable to traditional hyperinterpolation while avoiding its higher computational costs.
 Furthermore, we introduce a Lasso-based approach to improve the robustness of QMC hyperinterpolation against noise from sampling processes. Numerical experiments validate the efficacy of our proposed methods.
\end{abstract}



\section{Introduction} In this paper, we focus on the multi-dimensional unit cube \(\Omega:= [0, 1)^d\) for \(d \geq 1\). We investigate a novel high-quality approximation scheme for real-valued functions within the Hilbert space of square-integrable functions \(L_2(\Omega)\), termed quasi-Monte Carlo (QMC) hyperinterpolation, which addresses the limitations of traditional hyperinterpolation methods in high-dimensional spaces. 

Given a fixed index set \(I \subset \mathbb{Z}^d\) with finite cardinality \(|I|\), we aim to approximate a function \(f\) using a polynomial \(\mathcal{Q}_I f\) from the polynomial space \(\mathbb{P}_{I}(\Omega) := \text{span}\{q_{\bsh} : \bsh \in I\}\). The collection of \(q_{\bsh} : \Omega \to \mathbb{C}\) forms an orthonormal basis of \(L_2(\Omega)\) with the inner product defined as:

\[
\left<f, g\right> := \int_{\Omega} f(\bsx) \overline{g(\bsx)} \, d\lambda_d(\bsx),
\]
where \(\lambda_d\) is the \(d\)-dimensional Lebesgue measure, and the induced norm is given by \(\|f\|_{L_2} := \sqrt{\left<f, f\right>}\). For clarity, we refer to all spans of orthonormal bases as polynomials, even if they might be discontinuous. Notably, if \(I \subset \mathbb{Z}\), elements of \(I\) can be interpreted as the degree of a polynomial.

Hyperinterpolation, introduced by Sloan in \cite{SLOAN1995238}, approximates multivariate functions and is akin to a truncated Fourier expansion represented as a series of orthogonal polynomials for a specific discrete or continuous measure. Since its inception, various theoretical and practical aspects have been explored, establishing hyperinterpolation as a valuable alternative to traditional interpolation \cite{An_ran_2025,Lasso_hyper_2021,an2022quadrature,AN2024101789,caliari2007hyperinterpolation,dai2006GeneralizedHyperinterpolation,dai2013approximation,DeMarchi2009new,atkinson2009norm,reimer2012multivariate,SOMMARIVA2017110}. This method circumvents the challenge of selecting an optimal set of points, such as those with a low Lebesgue constant \cite{Wade2013hyperinterpolation}. Applications have been investigated in diverse 2D and 3D domains, including spheres, cubes, and less conventional shapes like polygons and circular sections \cite{SOMMARIVA2017110}.

We extend the classical notion of ``degree'' by the ``range of index set \(I\).''  Let \(I \subset \mathbb{Z}^d\) with finite cardinality \(|I|\) and define \(2I := \{ \mathbf{a} + \mathbf{b} \mid \mathbf{a} \in I, \mathbf{b} \in I \}\) as the Minkowski sum. We assume that the positive weights \((\omega_n)_{n=0}^{N-1}\) and quadrature points \((\bsx_n)_{n=0}^{N-1}\) satisfy the exactness assumption for the range \(2I\):

\begin{align}\label{hyper_exact_assump}
\sum_{n=0}^{N-1} \omega_n f(\bsx_n) = \int_{\Omega} f(\bsx) \, \rd \bsx \quad \forall f \in \mathbb{P}_{2I}.
\end{align}
In constructing hyperinterpolation, we define a ``discrete inner product" \cite{SLOAN1995238} as follows:

\[
\left<v, z\right>_N := \sum_{n=0}^{N-1} \omega_n v(\bsx_n) \overline{z(\bsx_n)}.
\]

\begin{definition}\label{def:hyper}
Assuming \eqref{hyper_exact_assump} holds, a hyperinterpolation of \(f \in L_2(\Omega)\) onto \(\mathbb{P}_{I}\) over the unit cube of range \(I\) is defined by

\begin{align}\label{Def_hyper}
\mathcal{L}_I f := \sum_{\bsh \in I} \left<f, q_{\bsh}\right>_N q_{\bsh}. 
\end{align}
\end{definition}
\sloppy The exactness requirement \eqref{hyper_exact_assump} poses significant challenges in constructing high-precision quadrature points, especially for high-dimensional problems.
 On the one hand, modifying the shape of the index set \(I\) leads to varying degrees, complicating the verification of the exactness assumption. For instance, a hypercube \(I(M) := \{\bsh \in \mathbb{Z}^d : \max_{1 \leq j \leq d} |h_j| \leq M\}\) corresponds to product degree, while a hyperbolic cross \(I(M) := \{\bsh \in \mathbb{Z}^d : \prod_{j=1}^d \max(1, |h_j|) \leq M\}\) corresponds to Zaremba cross degree \cite{Cools_1997}. On the other hand, the equality condition \eqref{hyper_exact_assump} means solving a complicated system. For example, constructing a Gauss quadrature rule in high-dimensional regions is computationally intensive \cite{Ryu_Boyd}. Additionally, finding a spherical \(t\)-design with large $t$ \cite{An_Well_conditioned_sherical} involves minimizing a large scale nonlinear, nonconvex problem. In order to enrich the range of candidate points for numerical integration in the realization of hyperinterpolation,
An and Wu \cite{AN2024101789} bypass the exactness assumption using the Marcinkiewicz–-Zygmund inequality \cite{filbir2011marcinkiewicz,Marcinkiewicz1937,mhaskar2001spherical} and propose the unfettered hyperinterpolation on the unit sphere. In \cite{AN2024101789}, extensive numerical experiments demonstrate that even scattered points can achieve the corresponding approximation, although the convergence speed is not very satisfactory.

To relax this constraint, we propose replacing exact quadrature rules \eqref{hyper_exact_assump} with a quasi-Monte Carlo (QMC) rule, which is an equally weighted quadrature rule over an \(N\)-element point set \((\bsx_n)_{n=0}^{N-1} \in \Omega\):

\begin{align}\label{QMC_rule}
\int_{\Omega} f(\bsx) \, \rd \bsx \approx \frac{1}{N} \sum_{n=0}^{N-1} f(\bsx_n).
\end{align}

To implement this, we introduce an additional assumption on the QMC rules, akin to the concept in \cite[Assumption 1]{AN2024101789}, stating a relaxed version of the exactness assumption:

\begin{assumption}\label{Assum:1}
Given a fixed index set \(I \subset \mathbb{Z}^d\) with finite cardinality \(|I|\), for  \(\eta \in (0, 1)\), there exists a \(N\) points QMC rule \eqref{QMC_rule}  such that

\begin{align}\label{QMC_assump}
\left|\frac{1}{N} \sum_{n=0}^{N-1} p^2(\bsx_n) - \int_{\Omega} p^2(\bsx) \, \rd \bsx\right| \leq \eta \int_{\Omega} p^2(\bsx) \, \rd \bsx \quad \forall p \in \mathbb{P}_{I}(\Omega).
\end{align}
Here, \(p^2\) denotes the product of \(p\) and its complex conjugate \(\overline{p}\).
\end{assumption}

Notice that this assumption resembles the strict exactness condition \eqref{hyper_exact_assump} when $\eta=0$. Allowing $\eta>0$ provides a way to control the approximation error instead of requiring strict equality.
 This assumption makes the construction of QMC hyperinterpolation feasible with various lattice rules \cite{dick2022lattice} and polynomial lattice rules \cite[Chapter 10]{Dick_Pillichshammer_2010}. With all weights \(w_j = 1/N\), the corresponding ``discrete inner product'' becomes:
\begin{align}\label{eq:discrete_inner_product}
\left<v, z\right>_N := \frac{1}{N} \sum_{n=0}^{N-1} v(\bsx_n) \overline{z(\bsx_n)}.
\end{align}
Then we introduce the concept of quasi-Monte Carlo hyperinterpolation, see Definition \ref{QMC_hyper_def}.
As considered in Assumption~\ref{Assum:1}, inheriting the essence of the Marcinkiewicz–-Zygmund property, QMC hyperinterpolation serves as a type of the unfettered hyperinterpolation introduced in \cite{AN2024101789}.

The numerical design of quadrature points for the QMC hyperinterpolation is straightforward; it requires that the quadrature rule \eqref{QMC_rule} converges to the integration of \(f\) as rapidly as possible, without additional construction for the equality \eqref{hyper_exact_assump} under different ranges \(I\). For the approximation quality of QMC hyperinterpolation, the relaxed assumption \eqref{QMC_assump} yields a practical error bound:

\[
\|\mathcal{Q}_I f - f\|_{L_2} \leq \left(\sqrt{1 + \eta} + 1\right) E_I(f) + \|\mathcal{Q}_I p^* - p^*\|_{L_2},
\]
where \(E_I(f) = \|f - p^*\|_\infty\) and \(p^* \in \mathbb{P}_I(\Omega)\) denotes the best approximation polynomial of \(f\) in the sense of the uniform norm. Here, the term $\|\mathcal{Q}_I p^* - p^*\|_{L_2}$ 
captures the aliasing error introduced by the relaxed quadrature assumption, which vanishes as $
N\to\infty $ under proper index selection.

Since QMC hyperinterpolation does not rely on the exactness assumption, it makes numerical design significantly simpler than for hyperinterpolation. In particular, as a concrete design, we provide component-by-component (CBC) algorithms for constructing rank-1 lattice rules \cite{kuo2003component}. Furthermore, with the aid of good rank-1 lattice rules, the worst \(L_2\) approximation error in Korobov space for \(\mathcal{Q}_I f\) can avoid the ``curse of dimensionality'' under some conditions. Finally,
to enhance robustness against sampling noise, we integrate the Lasso~\cite{Tibshirani1996RegressionLasso} into QMC hyperinterpolation. This regularization, motivated by its sparsity-promoting properties~\cite{Peter_Lasso_book_2011}, reduces the impact of noise while preserving convergence (Section~\ref{sec:lasso}).

In the sequel, Section 2 provides an overview of hyperinterpolation, including its mathematical foundations, key function spaces for error analysis, and applications. 
Section 3 introduces QMC hyperinterpolation, outlining its construction and $L_2$ error bounds. Section 4 presents specific algorithms for QMC hyperinterpolation using rank-1 lattice and polynomial lattice rules, emphasizing their advantage over traditional quadrature in avoiding the curse of dimensionality. Section 5 compares traditional hyperinterpolation with QMC hyperinterpolation, demonstrating the superior efficiency of the latter in high-dimensional settings through computational cost and accuracy analysis.
Section 6 enhances QMC hyperinterpolation’s noise robustness via a Lasso-based approach, improving approximation quality in noisy sampling.
Section 7 validates the theoretical results with numerical experiments, highlighting the effectiveness of QMC hyperinterpolation and its Lasso extension in practical applications.

\section{Background}\label{sec:background}
In this section, we give a brief review of hyperinterpolation and then introduce two function spaces. For simplicity, we denote
 $\mathbb{P}_{I}:=\mathbb{P}_{I}(\Omega)$.

\subsection{Hyperinterpolation on cube}
Similar to the original paper \cite{SLOAN1995238}, we derive the following results from Definition \ref{def:hyper}.
\begin{lemma}
  If $f\in\mathbb{P}_{I}$, then $\mathcal{L}_I f=f$.
\end{lemma}
For the approximation quality of $\mathcal{L}_I f,$  the following $L_2$ error bound holds.
\begin{theorem}
Given $f\in C(\Omega),$  let $\mathcal{L}_I f$ be defined by \eqref{Def_hyper}.
Then, we have
\begin{align}\label{hyper_L_2_error}
\Vert \mathcal{L}_I f\Vert_{L_2} \leq \Vert f\Vert_\infty\qquad \text{and}\qquad \Vert \mathcal{L}_I f-f\Vert_{L_2} \leq 2E_I(f),
\end{align}
where $E_I(f):=\inf_{{p\in\mathbb{P}_I}}\Vert f-p \Vert_{\infty}$ and $E_I( f )$ tends to 0 as $|I|$ approaches infinity.
\end{theorem}

For more details on hyperinterpolation, we refer the readers to \cite{an2022quadrature,SLOAN1995238}.

\subsection{Weighted Korobov space}\label{sub_Korobov}

We introduce the weighted Korobov spaces \cite{dick2022lattice}, a subspace of $L_2(\Omega)$, which is essential for the error analysis for the approximation of smooth periodic functions. Consider the trigonometric system $\{\exp(2\pi i \bsh \cdot \bsx),\;\bsh\in \mathbb{Z}^d,\;\bsx\in\Omega\}.$ Let $\alpha>1/2$ be a real number and $\bsgamma=(\gamma_1,\gamma_2,\ldots)$ be a sequence of real numbers with $0\leq \gamma_j\leq 1$ for all $j$.
For $\bsh \in \ZZ^{d}$, define
\begin{align*}
r_{\alpha, \bsgamma}(\bsh):=\prod_{\substack{j=1 \\ h_{j} \neq 0}}^{d} \frac{\left|h_{j}\right|^{\alpha}}{\gamma_{j}},
\end{align*}
where we write $r_{\alpha, \bsgamma}(\bszero)=1$
and we set $r_{\alpha, \bsgamma}(\bsh)=\infty$
if there exists an index $j\in \{1,\ldots,d\}$ such that $\gamma_j=0$ and $h_j\neq 0.$ Then the weighted Korobov space, denoted by $H_{d, \alpha, \bsgamma}$, is a reproducing kernel Hilbert space with the reproducing kernel
\[ K_{d, \alpha, \bsgamma}(\bsx, \bsy)=\sum_{\bsh \in \ZZ^{d}} \frac{\exp (2 \pi i \bsh \cdot(\bsx-\bsy))}{\left(r_{\alpha, \bsgamma}(\bsh)\right)^{2}}, \]
and the inner product
\[ \langle f, g\rangle_{d, \alpha, \bsgamma}=\sum_{\bsh \in \ZZ^{d}}\left(r_{\alpha, \bsgamma}(\bsh)\right)^{2} \hat{f}(\bsh) \overline{\hat{g}(\bsh)}, \]
where $\hat{f}(\bsh)$ denotes the $\bsh$-th Fourier coefficient of $f$, i.e.,
\[ \hat{f}(\bsh):= \int_{\Omega}f(\bsx)\exp(-2\pi i \bsh\cdot \bsx)\rd \bsx. \]
We denote the induced norm by $\|f\|_{d, \alpha, \bsgamma}:=\sqrt{\langle f, f\rangle_{d, \alpha, \bsgamma}}$.

Here, the parameter $\alpha>1/2$ measures the smoothness of periodic functions in the sense that, the larger $\alpha$ is, the Fourier coefficients of $f$ decay faster. The sequence of non-negative weights  $\gamma_{1}, \gamma_{2}, \ldots$ measure the relative importance of different variables \cite{sloan1998when}. In the case of $r_{\alpha,\bsgamma}(\bsh)=\infty$ with $\bsh\in\ZZ^d$ such that $h_j\neq 0,$ we set all such Fourier coefficients $\hat{f}(\bsh)$ and $\hat{g}(\bsh)$ are zero. It is worth noting that, when $\alpha$ is an integer, the Korobov space norm involves the mixed partial derivatives of $f$, see \cite[Section~2.1]{dick2022lattice}.
Moreover, it is known that this space $H_{d, \alpha, \bsgamma}$ is an algebra \cite[Appexndix~2]{Novak2004}.

\subsection{Weighted Walsh space}
Here we introduce another function space based on discontinuous basis functions.

Let $b$ be a fixed integer greater than $1$. For simplicity, we will later assume that $b$ is a prime number. Consider the system of Walsh functions in base $b$, denoted by $\{\wal_{\bsh}\mid \bsh\in \NN_0^d\}$, where each multivatiate Walsh function is defined by 
\[ \wal_{\bsh}(\bsx)=\prod_{j=1}^{d}\wal_{h_j}(x_j), \]
and the univariate Walsh function being
\[ \wal_h(x)=\exp\left( \frac{2\pi i}{b}\cdot \left( h_0x_1+h_1x_2+\cdots\right)\right),\]
where $h=h_0+h_1b+\cdots$ and $x=x_1/b+x_2/b^2+\cdots$ are the $b$-adic expansions of $h\in \NN_0$ and $x\in [0,1)$, respectively.
Since $h\in \NN_0$, its $b$-adic expansion is finite; that is, only finitely many digits are nonzero, and the expansion terminates with zeros in higher places.
For the $b$-adic expansion of $x\in [0,1)$, if two representations exist, we choose the one in which infinitely many digits differ from $b-1$.
Importantly, the Walsh functions form an orthonormal basis of $L_2(\Omega)$, see \cite[Appendix~A]{Dick_Pillichshammer_2010}.

For $h\in \NN$, let us denote the $b$-adic expansion differently by
\[ h = \eta_1 b^{c_1-1}+\eta_2b^{c_2-1}+\cdots + \eta_v b^{c_v-1},\]
with $\eta_1,\ldots,\eta_v\in \{1,\ldots,b-1\}$ and $c_1>c_2>\dots > c_v>0$. Then, we define the map $\mu_1: \NN_0\to \NN_0$ by $\mu_1(0)=0$ and $\mu_1(h)=c_1$ for $h\geq 1$.
Now, let $\alpha>1/2$ be a real number and $\bsgamma=(\gamma_1,\gamma_2,\ldots)$ be a sequence of real numbers with $0\leq \gamma_j\leq 1$ for all $j$.
For $\bsh \in \NN_0^d$, we define
\begin{align*}
\breve{r}_{\alpha, \bsgamma}(\bsh):=\prod_{\substack{j=1 \\ h_{j} \neq 0}}^{d} \frac{b^{\alpha \mu_1(h_j)}}{\gamma_{j}}.
\end{align*}

Then, the weighted Walsh space, denoted by $H_{d,\alpha,\bsgamma}^{\wal}$, is the reproducing kernel Hilbert space with reproducing kernel $K_{d,\alpha,\bsgamma}^{\wal}: [0,1]^d\times[0,1]^d\to \RR$ given by
\begin{align*}
K_{d,\alpha,\bsgamma}^{\wal}(\bsx,\bsy):=\sum_{\bsh\in \NN_0^d} \frac{\wal_{\bsh}(\bsx)\overline{\wal_{\bsh}(\bsy)}}{\breve{r}_{\alpha,\bsgamma}^2(\bsh)},
\end{align*}
and with inner product
\begin{align*}
\left<f,g  \right>^{\wal}_{d,\alpha,\bsgamma}:=\sum_{\bsh\in \NN_0^d} \breve{r}_{\alpha,\bsgamma}^2(\bsh) \breve{f}(\bsh)\overline{\breve{g}(\bsh)},
\end{align*}
where $\breve{f}(\bsh)$ denotes the $\bsh$-th Walsh coefficient of $f$:
\[ \breve{f}(\bsh):= \int_{\Omega}f(\bsx)\overline{\wal_{\bsh}(\bsx)} \rd \bsx. \]
We denote the induced norm by $\| f\|_{d,\alpha,\bsgamma}^{\wal}:=\sqrt{\left<f,f  \right>^{\wal}_{d,\alpha,\bsgamma} }.$
As shown in Appendix~\ref{app:Walsh_algebra}, this space $H_{d,\alpha,\bsgamma}^{\wal}$ is also an algebra.

\section{General framework of QMC hyperinterpolation}\label{sec:general}

\subsection{Definition of $\mathcal{Q}_I f$}
Regarding the significantly restrictive nature of \eqref{hyper_exact_assump}, which makes it impractical and sometimes impossible to obtain data on the desired quadrature points in practice, this paper aims to bypass this quadrature exactness assumption by replacing it with QMC rules. This leads to the introduction of QMC hyperinterpolation. However, most of the literature \cite{dick2004convergence,dick2022component,kritzer2019lattice,kuo2003component,sloan2002component} focuses on QMC rules to integrate smooth functions. In contrast, we focus on QMC rules that converge for the integrals of squared polynomials $p^2,$ where $p \in \mathbb{P}_{I}(\Omega).$  To achieve this, we impose the additional relaxed condition on the QMC rules as shown in Assumption~\ref{Assum:1}, which is similar to Marcinkiewicz–-Zygmund property in \cite[Assumption 1]{AN2024101789}.


Since all weights $w_j=1/N$ for QMC rules, the corresponding ``discrete inner product'' is given as in \eqref{eq:discrete_inner_product}, where the quadrature rule is a QMC rule \eqref{QMC_rule} that satisfies Assumption \ref{Assum:1}. Then  QMC hyperinterpolation $\mathcal{Q}_I f$ is defined as follows.
\begin{definition}[QMC hyperinterpolation]
Let $f\in L_2(\Omega),$ and $\mathcal{P}=(\bsx_n)_{n=0}^{N-1}$ be quadrature points satisfy Assumption \ref{Assum:1}. The quasi-Monte Carlo (QMC) hyperinterpolation of range $I$ is defined by
\begin{align}\label{QMChyper}
\mathcal{Q}_I f:=\sum_{\bsh \in I} \left<f,q_{\bsh}  \right>_{N} q_{\bsh}.
\end{align}
\end{definition}\label{QMC_hyper_def}
\begin{remark}
For the QMC rules that satisfy the exactness condition \eqref{hyper_exact_assump} for every $p\in \mathbb{P}_I(\Omega),$ Assumption \ref{Assum:1} is also satisfied. Therefore, some hyperinterpolations constructed using QMC rules can also be classified as QMC hyperinterpolation.
\end{remark}

Obviously, it holds that $\mathcal{Q}_I f\in\mathbb{P}_I(\Omega)$.
Moreover, $\mathcal{Q}_I f$ belongs to a reproducing kernel Hilbert space with the reproducing kernel
\begin{align*}
G_I(\bsx,\bsy):=\sum_{\bsh\in I} q_{\bsh}(\bsx)\overline{q_{\bsh}(\bsy)}
\end{align*}
in the sense that \cite{reimer2012multivariate}
\begin{align}\label{reproducing_property}
\left<g,   G_I(\cdot,\bsx)\right>=g(\bsx), \qquad \forall g\in \mathbb{P}_I(\Omega).
\end{align}

\subsection{$L_2$ error bound for $\mathcal{Q}_I f$}
With the aid of reproducing property \eqref{reproducing_property}, the construction of QMC hyperinterpolation implies the following theorem:
\begin{theorem}\label{thm_L_2_error}
Given $f\in L_2(\Omega),$ let $\mathcal{Q}_I f $ be  defined by \eqref{QMChyper}. Then
\begin{align}\label{QMC_L_2_oprator}
\|\mathcal{Q}_If\|_{L_2}\leq\sqrt{1+\eta}\|f\|_{\infty},
\end{align}
and
\begin{equation}\label{equ:L_2error}
\|\mathcal{Q}_I f-f\|_{L_2}\leq \left(\sqrt{1+\eta}+1 \right)E_I(f)+ \Vert \mathcal{Q}_I p^*-p^* \Vert_{L_2},
\end{equation}
where  $E_I(f)=\Vert f-p^*\Vert_\infty$ and  $p^*\in\mathbb{P}_I(\Omega)$ denotes the best approximation polynomial of $f$ in terms of the uniform norm.
\end{theorem}
\begin{proof}
For any $f\in L_2(\Omega)$, we have $\mathcal{Q}_I f\in\mathbb{P}_I(\Omega)$ and hence $\left\langle G_I(\cdot,\bsx_n),\mathcal{Q}_I f\right\rangle=\mathcal{Q}_I f(\bsx_n)$. Thus,

\begin{align*}
\left\langle \mathcal{Q}_I f,\mathcal{Q}_I f \right\rangle
& = \left\langle \frac{1}{N}\sum_{n=0}^{N-1} f(\bsx_n)G_I(\cdot,\bsx_n),\mathcal{Q}_I f\right\rangle = \frac{1}{N}\sum_{n=0}^{N-1} f(\bsx_n)\mathcal{Q}_I f(\bsx_n)\\
& \leq \left(\frac{1}{N}\sum_{n=0}^{N-1}f^2(\bsx_n)\right)^{1/2}\left(\frac{1}{N}\sum_{n=0}^{N-1}\left(\mathcal{Q}_If(\bsx_n)\right)^2\right)^{1/2}\\
&\leq \sqrt{1+\eta}\|f\|_{\infty}\|\mathcal{Q}_If\|_{L_2},
\end{align*}
where the first inequality is due to the Cauchy--Schwarz inequality and the second one holds by using \eqref{QMC_assump}. This estimate immediately implies the result \eqref{QMC_L_2_oprator}.

For the error bound \eqref{equ:L_2error}, we have 
\begin{align*}
\|\mathcal{Q}_I f-f\|_{L_2}
& = \|\mathcal{Q}_I(f-p)+(p-f)+(\mathcal{Q}_I p-p)\|_{L_2} \\
&\leq \|\mathcal{Q}_I(f-p)\|_{L_2} + \|f-p\|_{L_2}+\|\mathcal{Q}_I p-p\|_{L_2}\\
& \leq \sqrt{1+\eta}\|f-p\|_{\infty}+\|f-p\|_{\infty} + \|\mathcal{Q}_I p-p\|_{L_2}
\end{align*}
for any $p\in\mathbb{P}_I(\Omega).$ It follows, since this estimate holds for all polynomials $p$ in $\mathbb{P}_I(\Omega),$ that
\begin{align*}
\|\mathcal{Q}_I f-f\|_{L_2} \leq\left(\sqrt{1+\eta}+1\right)E_I(f) + \|\mathcal{Q}_I p^*-p^*\|_{L_2}.
\end{align*}
Thus, we have done.
\end{proof}

\begin{remark}
The $L_2$ upper bound of the QMC hyperinterpolation operator $\mathcal{Q}_I$ is larger than that of the hyperinterpolation operator $\mathcal{L}_I$ by a factor of $\eta \in (0,1)$, which accounts for the primary difference in $L_2$ approximation error between the QMC hyperinterpolation and the classical hyperinterpolation, given by $\Vert \mathcal{Q}_I p^*-p^*\Vert_{L_2}$. From the above error analysis, the approximation quality of QMC hyperinterpolation is slightly weaker than that of hyperinterpolation with the same range.
\end{remark}

In the following subsections, we give two classes of concrete quadrature points of $\mathcal{Q}_I f$ according to different polynomial spaces $ \mathbb{P}_I(\Omega).$  For short, we write
\begin{align}\label{int_error_criterion}
\text{err}_N(p,\mathbb{P}_{I}(\Omega)):&=\left|\int_{\Omega} p^2(\bsx) \rd \bsx-\frac{1}{N}\sum_{n=0}^{N-1} p^2(\bsx_n)\right|,\;\;p\in  \mathbb{P}_{I}(\Omega).
\end{align}

\subsection{Rank-1 lattice rules in weighted Korobov space}
For a periodic function $f\in L_2(\Omega),$ it is usual to use a multiple trigonometric polynomial to approximate it.  For a given index set $I,$ let  $\mathbb{P}_I(\Omega):=\left\{ \exp(2\pi i\bsh\cdot\bsx) ,\;\bsh\in I,\;\bsx\in\Omega\right\}.$  We can use the rank-1 lattice points to construct $\mathcal{Q}_I f.$ 
\begin{definition}[Rank-1 lattice point set]
    For $N\in \NN$ with $N\geq 2$, let $\bsz=(z_1,\ldots,z_d)\in \{1,\ldots,N-1\}^d$ be given. The rank-1 lattice point set $P_{N,\bsz}$ is defined by
    \[ P_{N,\bsz}:=\left\{ \bsx_n=\left(\left\{ \frac{nz_1}{N}\right\},\ldots,\left\{ \frac{nz_d}{N}\right\}\right) \,\mid \, 0\leq n<N\right\}.\]
    Here $\{x\}=x-\lfloor x \rfloor$ denotes the fractional part of a non-negative real number $x$.
\end{definition}
  The QMC rule based on a rank-1 lattice point set is called the rank-1 lattice rule. Note that the rank-1 lattice point set is characterized by the generating vector $\bsz.$ The following theorem illustrates the existence of a rank-1 lattice rule which satisfies Assumption \ref{Assum:1}.

\begin{theorem}\label{Example_A_existence}
Let an index set $I\subset \ZZ^d$ with finite cardinality \(|I|\), $\alpha>1/2$, and product weights $\boldsymbol{\gamma} = (\gamma_1, \gamma_2, \ldots) \in [0, 1]^{\mathbb{N}}$ be given.  There exists a positive integer $N^*$, which depends on $d$, $\alpha$ and $\bsgamma$, such that, for any prime $N\ge N^*$, the $N$-points rank-1 lattice rule with a suitable choice of the generating vector $\bsz \in \left\{1,\ldots,N-1 \right\}^d$ satisfies Assumption \ref{Assum:1} for all $p\in \mathbb{P}_{I}(\Omega).$ 
\end{theorem}

\begin{proof}
The existence follows the worst-case error in the weighted Korobov space. Here, the worst-case error is defined and bounded above by
\begin{align*}
e^{\text{wor}}( P_{N,\bsz}, H_{d,\alpha,\bsgamma}):&=\sup_{\substack{f\in H_{d,\alpha,\bsgamma}\\\Vert f\Vert_{d,\alpha,\bsgamma}\leq 1}}\left|\int_{\Omega} f(\bsx) \rd \bsx -\frac{1}{N}\sum_{n=0}^{N-1}f(\bsx_n) \right|\\
&\leq (N-1)^{-\tau}\left(-1+\prod_{j=1}^d \left( 1+2\gamma_j^{1/(2\tau)}\zeta(\alpha/\tau)   \right)\right)^\tau,
\end{align*}
where $\zeta$ denotes the Riemann Zeta function, defined by $ \zeta(\ell)=\sum_{t=1}^{\infty} t^{-\ell}$ and $\tau\in[1/2,\alpha).$  This upper bound is established in \cite[Theorem 2]{dick2006good}. Then for any polynomial $p\in \mathbb{P}_{I}(\Omega),$ we have 
\begin{align*}
\text{err}_N(p,\mathbb{P}_{I}(\Omega))\leq \frac{\Vert p^2\Vert_{d,\alpha,\bsgamma}}{{(N-1)}^\tau}\left(-1+\prod_{j=1}^d \left( 1+2\gamma_j^{1/(2\tau)}\zeta(\alpha/\tau)   \right)\right)^\tau.
\end{align*}
As the weighted Korobov spaces are algebras for $\alpha>1/2$ \cite[Appendix~2]{Novak2004}, that is 
\begin{align*}
\Vert p^2\Vert_{d,\alpha,\bsgamma}\leq C(d)\Vert p\Vert_{d,\alpha,\bsgamma}^2,
\end{align*}
with 
$
C(d)=2^{d\;\max(1,\alpha)}\prod_{j=1}^d(1+2\gamma_j^2\zeta(2\alpha))^{1/2},
$ it holds that 
\begin{align*}
\Vert p\Vert_{d,\alpha,\bsgamma}^2\leq C(d)\left(\max_{\bsh\in I} r_{\alpha,\bsgamma}^2(\bsh)\right)\Vert p\Vert_{L_2}^2.
\end{align*}
This implies that there exist $N^*\in \NN$ such that Assumpiton~\ref{Assum:1} is satisfied for any $N\ge N^*$. 
\end{proof}

In the case the rank-1 lattice rule satisfying Assumption \ref{Assum:1}, our approximation scheme \eqref{QMChyper} becomes 
\begin{align*}
\mathcal{Q}_I f:=\sum_{\bsh\in I}\left(\frac{1}{N}\sum_{n=0}^{N-1} f\left( \left\{\frac{n}{N} \bsz \right\} \right)\exp(-2\pi i n \bsh\cdot \bsz/N) \right)\exp(2\pi i \bsh \cdot\bsx),\;\;\bsx\in\Omega.
\end{align*}

\subsection{Polynomial lattice rules in weighted Walsh space}
For non-periodic functions $f\in L_{2}(\Omega)$, one can use the system of multivariate Walsh functions for approximation. Given a fixed index set $I,$ define
\[  \mathbb{P}_I^{\wal}(\Omega):=\mathrm{span}\left\{ \wal_{\bsh},\;\bsh\in I\right\}. \]
Then we can construct $\mathcal{Q}_I f$  by polynomial lattice rules. The polynomial lattice rule is a QMC rule based on a polynomial lattice point set, and in particular, the rank-1 polynomial lattice point set is analogous to the rank-1 lattice point set, but is based on linear algebra over finite fields.  

Let $b$ be a prime number and $\FF_b$ denote the finite field with $b$ elements. We identify $\FF_b$ with the set $\{0,1,\ldots,b-1\}$, equipped with addition and multiplication modulo $b$. We denote by $\FF_b[x]$ the ring of polynomials over $\FF_b$, and by $\FF_b((x^{-1}))$ the field of formal Laurent series over $\FF_b$.
Then the rank-1 polynomial lattice point set is defined as follows.

\begin{definition}
    Let $m\in \NN$, and let $p\in \FF_b[x]$ and $\bsq =(q_1,\ldots,q_d)\in (\FF_b[x])^d$ be such that $\deg(p)=m$ and $\deg(q_j)< m$ for all $j$. We identify an integer $h\in \NN_0$, whose $b$-adic expansion is given by $h=\eta_0+\eta_1b+\ldots$, with the polynomial $h(x)=\eta_0+\eta_1 x+\cdots\in \FF_b[x]$. The rank-1 polynomial lattice point set with modulus $p$ and generating vector $\bsq$, denoted by $P(p,\bsq)$, is the set of $b^m$ points defined by
    \[ \bsx_h=\left( \nu_m\left(\frac{h(x)q_1(x)}{p(x)} \right),\ldots, \nu_m\left(\frac{h(x)q_d(x)}{p(x)} \right)\right)\in \Omega,\]
    where the map $\nu_m: \FF_b((x^{-1}))\to [0,1)$ is defined by
    \[ \nu_m\left( \sum_{i=w}^{\infty}a_ix^{-i}\right) = \sum_{i=\max\{1,w\}}^{m}a_ib^{-i}.\]
\end{definition}

Polynomial lattice point sets are an important class of so-called \emph{digital nets}, a family of QMC point sets, see \cite{Dick_Pillichshammer_2010}. 
With this fact in mind, we have a different way to represent rank-1 polynomial lattice point sets using the generating matrices over $\FF_b$.
Similar to the rank-1 lattice rules in weighted Korobov space, we have the following existence theorem for Assumption \ref{Assum:1}.

\begin{theorem}\label{Example_C_existence}
Let an index set $I\subset \NN_0^d$ with finite cardinality \(|I|\), $\alpha>1/2$, and product weights $\boldsymbol{\gamma} = (\gamma_1, \gamma_2, \ldots) \in [0, 1]^{\mathbb{N}}$ be given. There exists a positive integer $m^*$, which depends on $d$, $\alpha$ and $\bsgamma$, such that, for any integer $m\ge m^*$, the $b^m$-points rank-1 polynomial lattice rule with a suitable choice of modulus $p\in \FF_b[x]$ and generating vector $\bsq\in (\FF_b[x])^d$ satisfies Assumption \ref{Assum:1} for all $p\in \mathbb{P}^{\wal}_{I}(\Omega).$ 
\end{theorem}

\begin{proof}
    A similar argument made in \cite{DP2005} shows that there exists a polynomial lattice rule with modulus $p\in \FF_b[x]$ and generating vector $\bsq \in (\FF_b[x])^d$ such that $p$ is irreducible with $\deg(p)=m$ and $\deg(q_j)<m$, for which the worst-case error in the weighted Walsh space, $H_{d,\alpha,\bsgamma}^{\wal}$ is bounded by
    \begin{align*}
        e^{\wor}(P(p,\bsq), H_{d,\alpha,\bsgamma}^{\wal}) & := \sup_{\substack{f\in H_{d,\alpha,\bsgamma}^{\wal}\\ \|f\|_{d,\alpha,\bsgamma}^{\wal}\le 1}}\left| \int_{[0,1)^d}f(\bsx)\rd \bsx - \frac{1}{b^m}\sum_{h=0}^{b^m-1}f(\bsx_h) \right|\\
        & \: \leq \left( \frac{2}{b^m-1}\sum_{\bsk\in \NN_0^d\setminus \{\bszero\}}\frac{1}{(\breve{r}_{\alpha, \bsgamma}(\bsk))^{2\lambda}} \right)^{1/\lambda}\\
        & \: = \left( \frac{2}{b^m-1}\left(-1+\prod_{j=1}^{d}\left(1+\gamma_{j}^2\frac{b-1}{b^{2\alpha\lambda}-b}\right)\right)\right)^{1/\lambda},
    \end{align*}
    for any $\lambda\in (1/(2\alpha),1]$.
    
    As shown in Appendix~\ref{app:Walsh_algebra}, the weighted Walsh space with $\alpha>1/2$ is an algebra, i.e., we have
    \[ \| p^2\|_{d,\alpha,\bsgamma}^{\wal}\leq \breve{C}(d) \left(\|p\|_{d,\alpha,\bsgamma}^{\wal}\right)^2 \quad \text{with $\breve{C}(d)=2^d \prod_{j=1}^{d}\left(1+\gamma_{j}^2\frac{b-1}{b^{2\alpha}-b}\right)^{1/2}$,} \]
    for any $p\in H_{d,\alpha,\bsgamma}^{\wal}$. Thus, it holds for any $p\in \mathbb{P}_I^{\wal}(\Omega)$ that
    \begin{align*}
        \text{err}_{b^m}(p,\mathbb{P}^{\wal}_{I}(\Omega)) & \leq \| p^2\|_{d,\alpha,\bsgamma}^{\wal}e^{\wor}(P(p,\bsq), H_{d,\alpha,\bsgamma}^{\wal})\\
        & \leq \breve{C}(d) \left(\|p\|_{d,\alpha,\bsgamma}^{\wal}\right)^2\left( \frac{2}{b^m-1}\left(-1+\prod_{j=1}^{d}\left(1+\gamma_{j}^2\frac{b-1}{b^{2\alpha\lambda}-b}\right)\right)\right)^{1/\lambda}\\
        & \leq \breve{C}(d) \left(\max_{\bsh\in I} \breve{r}_{\alpha,\bsgamma}^2(\bsh)\right)\| p\|_{L_2}^2\\
       &\qquad\qquad\times \left( \frac{2}{b^m-1}\left(-1+\prod_{j=1}^{d}\left(1+\gamma_{j}^2\frac{b-1}{b^{2\alpha\lambda}-b}\right)\right)\right)^{1/\lambda}.
    \end{align*}
    This proves the existence of $m^*\in \NN$ such that Assumption~\ref{Assum:1} holds for any $m\ge m^*$.
\end{proof}

In this case, our approximation scheme \eqref{QMChyper} becomes 
\begin{align*}
\mathcal{Q}_I f(\bsy) 
:=\sum_{\bsh\in I} 
& \left( \frac{1}{b^m} \sum_{m=0}^{b^m-1} 
f\left( \nu_m\left( \frac{n(x)\bsq(x)}{p(x)} \right) \right) 
\overline{\wal_{\bsh}\left( \nu_m\left( \frac{n(x)\bsq(x)}{p(x)} \right) \right)} 
\right) \wal_{\bsh}(\bsy).
\end{align*}

\subsection{Two-dimensional examples of quadrature points}
We present special examples of two-dimensional quadrature point sets for QMC hyperinterpolation, illustrated in Figure \ref{Figure_1}. In this figure, the left panel shows a rank-1 lattice point set of size \(F_n\) with $n=11$, generated by the vector \((1, F_{n-1})\). Here, \(F_n\) denotes the \(n\)-th Fibonacci number defined recursively as
\[
F_1 = 1, \quad F_2 = 1,
\]
\[
F_n = F_{n-1} + F_{n-2}, \quad n \geq 3.
\]
The right panel shows a polynomial lattice point set analogue of the Fibonacci lattice point set.
The modulus $p$ is given by $F_n(x)\in \FF_b[x]$ and the generating vector $\bsq$ is by $(1,F_{n-1}(x))\in (\FF_b[x])^2$, both with $n=7$.
Here $F_n(x)$ denotes the \(n\)-th Fibonacci polynomial over $\FF_b$, defined recursively as
\[
F_1(x) = 1, \quad F_2(x) = x,
\]
\[
F_n(x) = x F_{n-1}(x) + F_{n-2}(x), \quad n \geq 3.
\]

\begin{figure}[tbp]
  \centering
  \includegraphics[width=0.8\linewidth]{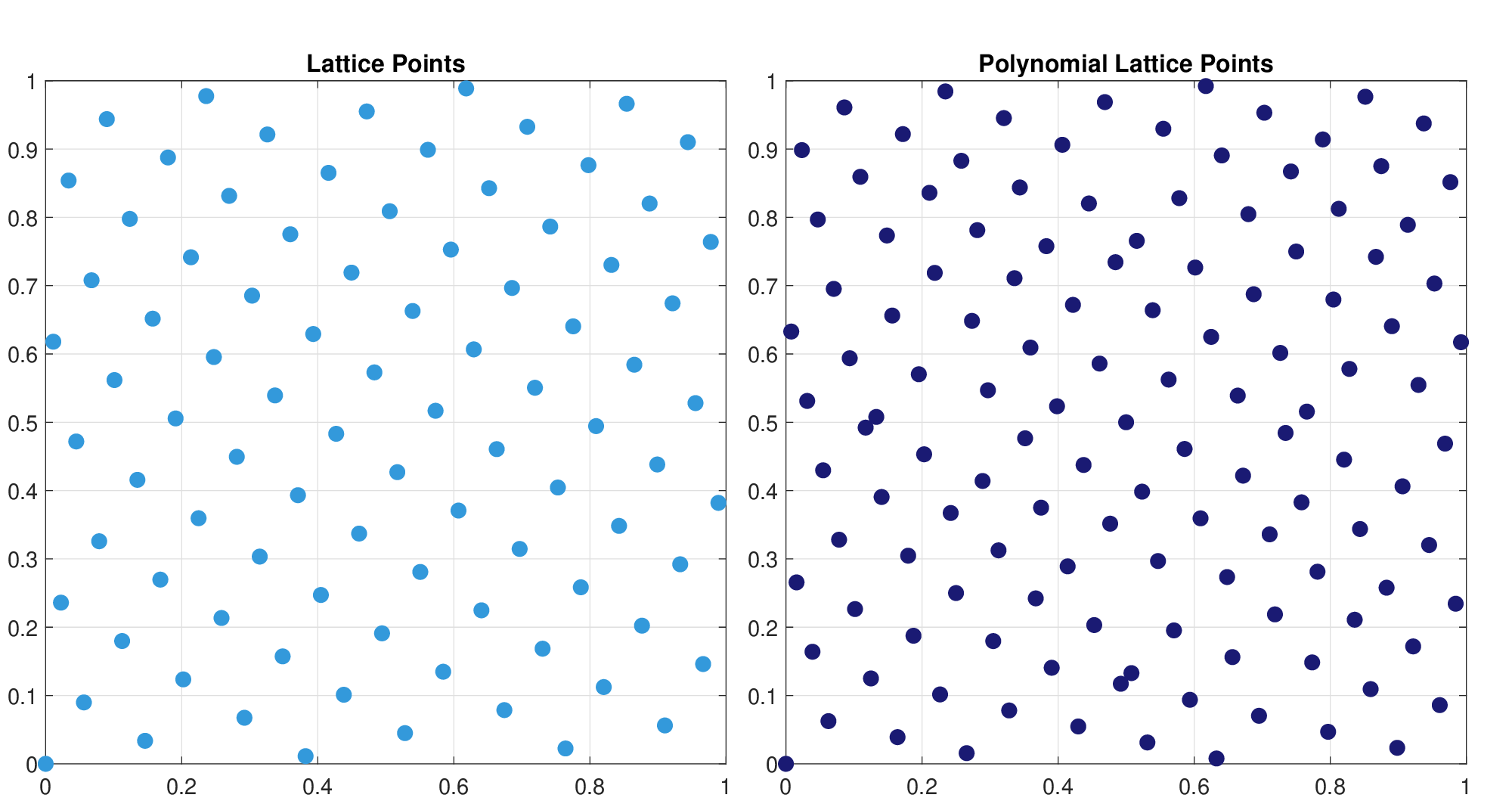}\\
\caption{Two-dimensional QMC points: (left) rank-1 lattice of size $89$ based on Fibonacci numbers;  (right) Fibonacci polynomial lattice of size 
$2^7=128.$  } \label{Figure_1}
\end{figure}

\section{Construction algorithms of QMC hyperinterpolation}\label{sec:construction}
In this section, we give concrete construction algorithms for QMC hyperinterpolation with rank-1 lattice rules and polynomial lattice rules. Additionally, we provide $L_2$ error estimates for $\mathcal{Q}_I f$ under certain smoothness assumptions, as the construction algorithms are usually based on such error analysis.  Regarding the choice of the index set $I$, it is desirable to make the number of quadrature points in Assumption \ref{Assum:1} as small as possible to keep the computational cost low. 

In case of rank-1 lattice rules, we know from the proof of Theorem  \ref{Example_A_existence} that the special quantity $r_{\alpha,\bsgamma}(\bsh)$ over $\mathbb{Z}^d$ would effect the size of $N$ in Assumption \ref{Assum:1}. Hence, for a real number $M > 0,$ we constrain the index set $I$ in the hyperbolic cross set
\begin{align*}
\mathcal{A}_{d}(M)=\{\bsh\in \mathbb{Z}^d:\left(r_{\alpha,\bsgamma}(\bsh)  \right)^2 \leq M \}.
\end{align*}
Figure \ref{Figure_2} shows a two-dimensional example of the hyperbolic cross set with different parameter settings. From this figure, we can see another advantage of the hyperbolic cross set, that is, it can make the size of the index set $I$ small. To this end, we bound
the cardinality of the hyperbolic cross set. For any $\lambda>1/(2\alpha),$ the cardinality of $\mathcal{A}_d(M)$ can be estimated as following (see \cite[Chaper 13]{dick2022lattice}).
\begin{align}\label{dimension_A_d}
|\mathcal{A}_d(M)|\leq M^{\lambda} \prod_{j=1}^d(1+2\gamma_j^{2\lambda} \zeta(2\alpha \lambda)).
\end{align}
Obviously, the cardinality of $|\mathcal{A}_d(M)|$ can be independent of dimension $d$ under some summability conditions on the weights $\bsgamma$. For the convenience of theoretical analysis, we always let the index set $I$ be the hyperbolic cross set $\mathcal{A}_d(M)$ in the rest of the paper. 

\begin{figure}[tbp]
  \centering
  \includegraphics[width=1\linewidth]{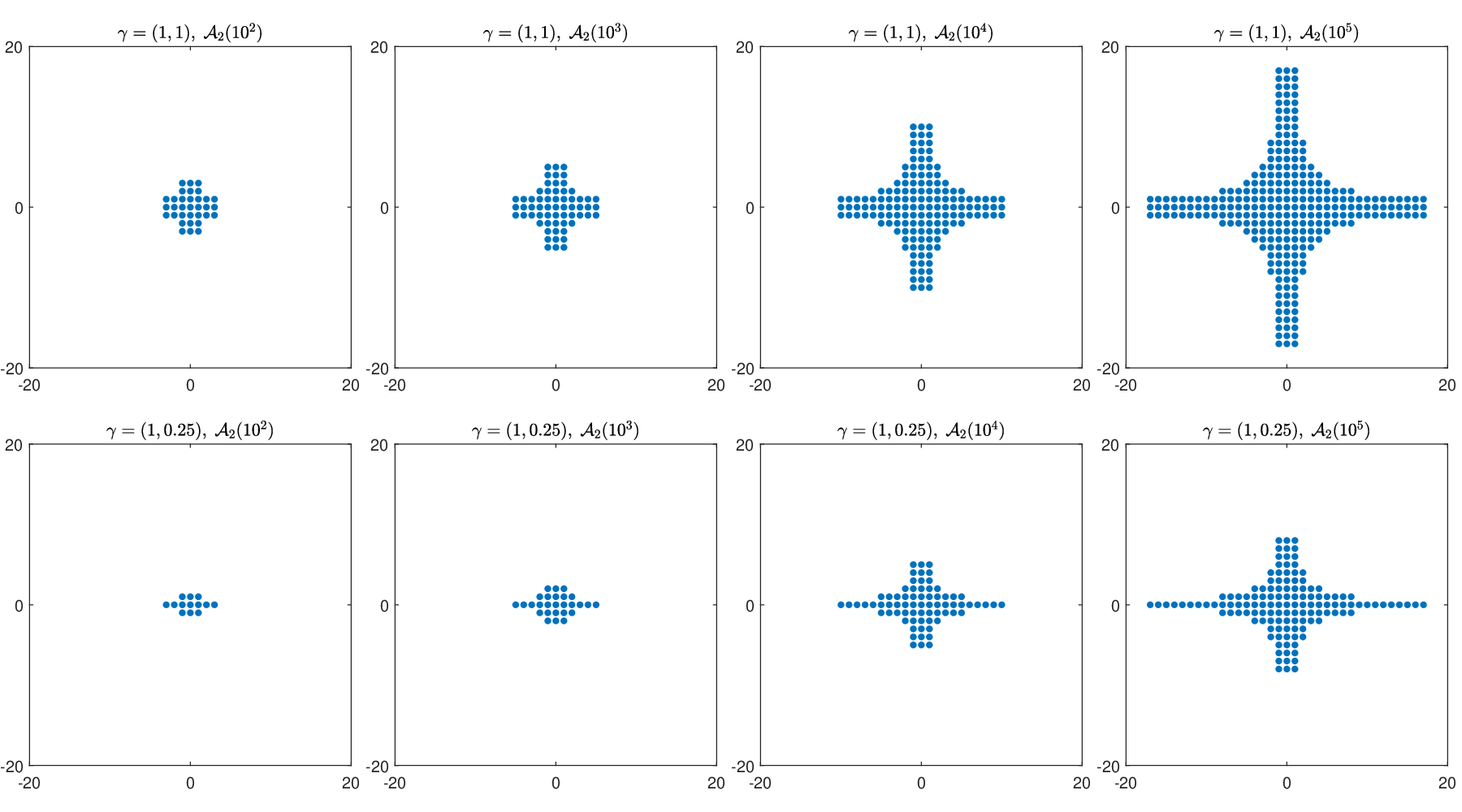}\\
\caption{Two-dimensional hyperbolic cross set for $\alpha=2.$ } \label{Figure_2}
\end{figure}

\subsection{CBC Constructions for Rank-1 Lattice Rules}
Now we give concrete computer search algorithms for a good rank-1 lattice point set $P_{N,\bsz}$ for $\mathcal{Q}_If$, so that it satisfies Assumption \ref{Assum:1}. Recall that the $P_{N,\bsz}$ is fully characterized by the generating vector $\bsz\in \mathbb{Z}^d.$  Therefore, it suffices to search for the generating vector $\bsz$ from
\begin{align}\label{CBC_serach_range}
G_{d}(N):=\left\{z\in\left\{1,2,\ldots,N-1  \right\}\;:\;\gcd(\bsz,N)=1  \right\}^d,
\end{align}
where the restriction that the components of the generating vector $\bsz$ are all coprime with $N$ would let every projection of a lattice point set have $N$ distinct values. When $N$ is prime, the number of all possible candidates in $G_d(N)$ is $(N-1)^d.$  Hence, the exhaustive search strategy for generating vector $\bsz$ is infeasible as the great computation costs of $(N-1)^d$ for large $N$ and $d.$

Nowadays, the most widely used computer search algorithm for the generating vector $\bsz$ of a rank-1 lattice point set is the component-by-component construction, or for short, CBC construction, which was first proposed by Korobov \cite{Korobov1963} and reinvented by Sloan and Reztsov \cite{sloan2002component}. We outline the general principle of CBC construction in Algorithm~\ref{cbc_principle}.
Here we remark that the selection criterion can be designed by minimizing an error criterion, or by satisfying congruence equations, and so on. Moreover, if there exist several distinct elements $z_s\in G_1(N)$ that satisfy the selection criterion $R$, it is allowed either to choose any of them arbitrarily or to make some rule so that the choice is unique.

\begin{algorithm}[t]
\caption{CBC construction principle}
\label{cbc_principle}
Let $N$ be a prime number, let $d\in\mathbb{N}.$
Construct a generating vector $\bsz=(z_1,\ldots,z_d)\in G_d(N)$, where $G_d(N)$ is defined in \eqref{CBC_serach_range}, as follows.
\begin{enumerate}
\item Choose a sufficiently good one-dimensional generator $g_1\in G_1(N),$ in most cases, $g_1$ is set to 1.
\item \textbf{For} $s$ from $2$ to $d$ do the following:\\
Assume that $\{z_1,\ldots,z_{s-1}\}\in G_{s-1}(N)$ have already been found. Choose $z_{s}\in G_1(N)$ by some selection boolean criterion $R,$ i.e., pick up $z_s$ such that
\begin{align*}
R(z_1,\ldots,z_{s-1},z_s)=\text{true}.
\end{align*}
\textbf{end for}
\item Set $\bsz=(z_1,\ldots,z_{d}).$
\end{enumerate}
\end{algorithm}

Note that the size of the search space of the generating vector $\bsz$ in Algorithm \ref{cbc_principle} is $d (N-1),$ which has a significant reduction as compared to the exhaustive search with a full space of size $(N-1)^d.$ Furthermore, the CBC construction is extensible; that is, additional components $z_{d+1}$ can be added to the already-obtained vector $\bsz\in \mathbb{Z}^d$ by running one loop in the CBC algorithm. However, note that we also need to consider the extra computation costs of the selection boolean criterion $R.$

Since the key of constructing the generating vector $\bsz$ for QMC hyperinterpolation is Assumption \ref{Assum:1}, we can design the selection boolean criterion $R$ by minimizing the integration error criterion \eqref{int_error_criterion}. The following character property for the dual lattice is necessary to derive the analytic form of \eqref{int_error_criterion}.
\begin{definition}[Dual lattice]
    For $N\in \NN$ with $N\geq 2$ and a generating vector $\bsz\in \{1,\ldots,N-1\}^d$, the set
    \[ P_{N, \bsz}^{\perp}:=\left\{\bsh \in \ZZ^{d} \mid \bsh \cdot \bsz \equiv 0 \pmod N\right\} \]
    is called the dual lattice of the rank-1 lattice point set $P_{N, z}$.
\end{definition}
Then the following property of rank-1 lattice rules holds.
\begin{lemma}[Character property]\label{lem:character}
    For $N\in \NN$ with $N\geq 2$ and a generating vector $\bsz\in \{1,\ldots,N-1\}^d$,
    \[ \frac{1}{N}\sum_{n=0}^{N-1} \exp \left(2 \pi i \bsh \cdot \bsx_n\right)= \begin{cases}1 & \text{if $\bsh \in P_{N, \bsz}^{\perp}$,} \\ 0 & \text{otherwise,}\end{cases} \]
    holds for any vector $\bsh \in \ZZ^{d}$.
\end{lemma}
 With the help of character property, we can deduce the analytic form of integration error \eqref{int_error_criterion}. Denote $p\in  \mathbb{P}_I (\Omega)\setminus\{\boldsymbol{0}\}$ by $p(\bsx)=\sum_{\bsh\in I} a_{\bsh} \exp(2\pi i\bsh \cdot \bsx)$. Then we have 
\[ \int_{\Omega}p^2(\bsx)\rd \bsx  = \sum_{\bsh\in I}|a_{\bsh}|^2\]
and
\begin{align*}
\frac{1}{N}\sum_{n=0}^{N-1} p^2(\bsx_n)&=\frac{1}{N}\sum_{n=0}^{N-1}\left(\sum_{\bsh \in I} a_{\bsh}\exp(2\pi i \bsh\cdot \bsx_n )\right)^2\\
&=\sum_{\bsh'\in I}\sum_{\bsh''\in I} a_{\bsh'}\overline{a_{\bsh''}}\left(\frac{1}{N} \sum_{n=0}^{N-1} \exp(2\pi i (\bsh'-\bsh'')\cdot\bsx_n)\right)\\
&=\sum_{\bsh'\in I} a_{\bsh'}\sum_{\substack{\bsh''\in I\\ \bsh'-\bsh''\in P_{N,\bsz}^{\perp}}} \overline{a_{\bsh''}},
\end{align*}
so that the signed integration error is equal to
\begin{align}
\frac{1}{N}\sum_{n=0}^{N-1} p^2(\bsx_n)-\int_{\Omega}p^2(\bsx)\rd \bsx =\sum_{\bsh'\in I} a_{\bsh'}\sum_{\substack{\bsh''\in I\\ \bsh'-\bsh''\in P_{N,\bsz}^{\perp}\backslash\{ \boldsymbol{0}  \}}} \overline{a_{\bsh''}}. \label{integration_error}
\end{align}

To extract the constant $\eta$ in Assumption \ref{Assum:1}, considering the size of hyperbolic cross set \eqref{dimension_A_d}, we have the first integration error estimates.

\begin{lemma}\label{lem:assumption_error_bounds}
For the rank-1 lattice rule, we have the following error estimation.
\begin{align}\label{lattice_integration_error}
\emph{err}_N(p,\mathbb{P}_I(\Omega))\leq \sqrt{|I|}\Vert p \Vert_{d,\alpha,\bsgamma}^2  R_{N,\alpha,d,\bsgamma}(\bsz) \prod_{j=1}^d \max(1,2^{\alpha}\gamma_j),
\end{align}
and
\begin{align}\label{lattice_integration_error_R}
\emph{err}_N(p,\mathbb{P}_I(\Omega))\leq \Vert p\Vert_{d,\alpha,\bsgamma}^2 S_{N,\alpha,d,\bsgamma}(\bsz),
\end{align}
where we write
\begin{align}
R^2_{N,\alpha,d,\bsgamma}(\bsz) & =\sum_{\bsh \in P_{N,\bsz}^{\perp} \backslash \{\boldsymbol{0} \}}\frac{1}{r_{\alpha,\bsgamma}^2(\bsh)},\qquad \text{and} \label{quantity} \\
S^2_{N,\alpha,d,\bsgamma}(\bsz)
& =\sum_{\bsh^* \in P_{N,\bsz}^{\perp}\backslash \{ \boldsymbol{0}  \}}\sum_{\bsh'\in\mathbb{Z}^d} \frac{1}{r_{\alpha,\bsgamma}^2(\bsh')r_{\alpha,\bsgamma}^2(\bsh^*+\bsh')}.\notag
\end{align}
\end{lemma}

\begin{proof}
 Note that the polynomial $p$ is infinitely smooth, thus $p\in H_{d,\alpha,\bsgamma}$ for any $\alpha>1/2$, and it follows from \eqref{integration_error} that
\begin{align*}
\left| \int_{\Omega} p^2(\bsx) \rd \bsx-\frac{1}{N}\sum_{n=0}^{N-1} p^2(\bsx_n)  \right|\leq \sum_{\bsh' \in I} |a_{\bsh'}|\sum_{\substack{\bsh''\in I \\ \bsh'-\bsh'' \in P_{N,\bsz}^{\perp}\backslash \{ \boldsymbol{0} \}}}| a_{\bsh''}|.
\end{align*}
Then, by using the Cauchy–-Schwarz inequality twice, we obtain
\begin{align*}
\text{err}_N(p,\mathbb{P}_I(\Omega))&\leq  R_{N,\alpha,d,\bsgamma}(\bsz)\sum_{\bsh' \in I} |a_{\bsh'}|\left(\sum_{\substack{\bsh''\in I \\ \bsh'-\bsh'' \in P_{N,\bsz}^{\perp}\backslash\{\boldsymbol{0}\}}}| a_{\bsh''}|^2r^2_{\alpha,\bsgamma}(\bsh'-\bsh'')\right)^{1/2} \\
&\leq  R_{N,\alpha,d,\bsgamma}(\bsz)\\
&\qquad\times\sum_{\bsh' \in I} |a_{\bsh'}|\left(\sum_{\substack{\bsh''\in I \\ \bsh'-\bsh'' \in P_{N,\bsz}^{\perp}\backslash\{\boldsymbol{0}\}}}\left( a_{\bsh''} r_{\alpha,\bsgamma}(\bsh'')\right)^2\frac{r^2_{\alpha,\bsgamma}(\bsh'-\bsh'')}{r_{\alpha,\bsgamma}^2(\bsh'')}\right)^{1/2} \\
&\leq \Vert p\Vert_{d,\alpha,\bsgamma} R_{N,\alpha,d,\bsgamma}(\bsz)\sum_{\bsh'\in I}|a_{\bsh'}|\left( \sup_{\substack{\bsh''\in I\\ \bsh''-\bsh' \in P_{N,\bsz}^{\perp}\backslash\{\boldsymbol{0}\}}} \frac{r_{\alpha,\bsgamma}^2(\bsh'-\bsh'')} {r_{\alpha,\bsgamma}^2(\bsh'')}\right)^{1/2} \\
&\leq \Vert p\Vert_{d,\alpha,\bsgamma}^2 R_{N,\alpha,d,\bsgamma}(\bsz)\\
&\qquad\times\left( \sum_{\bsh'\in I}\frac{1}{r^2_{\alpha,\bsgamma}(\bsh')}\sup_{\substack{\bsh''\in \mathbb{Z}^d \\ \bsh''-\bsh' \in P_{N,\bsz}^{\perp}\backslash\{\boldsymbol{0}\}}} \frac{r_{\alpha,\bsgamma}^2(\bsh'-\bsh'')} {r_{\alpha,\bsgamma}^2(\bsh'')}\right)^{1/2}.
\end{align*}
Let $\bsh^*=\bsh''-\bsh'\in P_{N,\bsz}^{\perp}\backslash\{\boldsymbol{0}\}.$ For any choice of $\bsh'\in I,$  from \cite[Chapter 13]{dick2022lattice}, it is true that
\begin{align*}
\frac{r_{\alpha,\bsgamma}^2(\bsh'-\bsh'')} {r_{\alpha,\bsgamma}^2(\bsh'')}=\frac{r_{\alpha,\bsgamma}^2(\bsh^*)}{r^2_{\alpha,\bsgamma}(\bsh^*+\bsh')}\leq r_{\alpha,\bsgamma}^2(\bsh')\prod_{j=1}^d \max(1,2^{2\alpha}\gamma_j^2),
\end{align*}
which leads to \eqref{lattice_integration_error}. On the other hand, we can get rid of the size of index set $I$ by introducing a new quantity. By rewriting the integration error \eqref{int_error_criterion}, we have
\begin{align*}
\text{err}_N(p,\mathbb{P}_I(\Omega))&\leq \sum_{\bsh' \in I} |a_{\bsh'}|\sum_{\substack{\bsh''\in I \\ \bsh'-\bsh'' \in P_{N,\bsz}^{\perp}\backslash \{ \boldsymbol{0}  \}}}\left(| a_{\bsh''}|r_{\alpha,\bsgamma}(\bsh'')\right) \frac{1}{r_{\alpha,\bsgamma}(\bsh'')}.
\end{align*}
Again, by using the Cauchy–-Schwarz inequality twice, we have
\begin{align*}
\text{err}_N(p,\mathbb{P}_I(\Omega)) &\leq \Vert p\Vert_{d,\alpha,\bsgamma}\sum_{\bsh'\in I} |a_{\bsh'}|\left(\sum_{\substack{\bsh''\in \mathbb{Z}^d \\ \bsh'-\bsh'' \in P_{N,\bsz}^{\perp}\backslash \{ \boldsymbol{0}  \}}} \frac{1} {r_{\alpha,\bsgamma}^2(\bsh'')} \right)^{1/2}\\
&\leq \Vert p\Vert_{d,\alpha,\bsgamma}^2\left(  \sum_{\bsh'\in\mathbb{Z}^d}\frac{1}{r_{\alpha,\bsgamma}^2(\bsh')}\sum_{\bsh^* \in P_{N,\bsz}^{\perp}\backslash \{ \boldsymbol{0}  \}} \frac{1}{r_{\alpha,\bsgamma}^2(\bsh^*+\bsh')} \right)^{1/2}
\end{align*}
and therefore obtain \eqref{lattice_integration_error_R}.
\end{proof}

Using the definition of the hyperbolic cross set, we can distinguish $\Vert p\Vert_{L_2}^2$ for $p\in \mathbb{P}_{I}(\Omega),$ from $\Vert p\Vert_{d,\alpha,\bsgamma}^2,$ that is
\begin{align*}
\Vert p\Vert_{d,\alpha,\bsgamma}^2=\sum_{\bsh\in I} |a_{\bsh}|^2 r^2_{\alpha,\bsgamma}(\bsh)\leq M \Vert p\Vert_{L_2}^2.
\end{align*}
Hence, the integration error $\text{err}_N(p,\mathbb{P}_I(\Omega))$ is characterized by special quantities $R_{N,\alpha,d,\bsgamma}(\bsz)$ and $S_{N,\alpha,d,\bsgamma}(\bsz).$ In fact, under suitable choices of $\bsz$ (depending on $N$), both $R_{N,\alpha,d,\bsgamma}(\bsz)$ and $S_{N,\alpha,d,\bsgamma}(\bsz)$ tend to zero as $N\rightarrow \infty$, thereby ensuring that Assumption \ref{Assum:1} is eventually satisfied with a suitable number of function evaluations  $N.$ However, a good generating vector would also lead to a faster convergence rate of them, and then it is natural to construct such a good generating vector $\bsz$ by 
CBC construction with selection criterion that minimizes either $R_{N,\alpha,s,\bsgamma}(\bsz)$ or $S_{N,\alpha,s,\bsgamma}(\bsz)$ from the range $G_1(N)$ in Algorithm \ref{cbc_principle}. 

Note that
there is a concise computable form of the criterion \eqref{quantity}  when \(\alpha\) is a natural number:
\[
R^2_{N,\alpha,s,\bsgamma}(\bsz) = \frac{1}{N} \sum_{\bsx \in P_{N,\bsz}} \left[ -1+\prod_{j=1}^{s}\left(1+\gamma_j^2   (2\pi)^{2\alpha} \frac{(-1)^{\alpha+1}}{(2\alpha)!} B_{2\alpha}(x_j) \right) \right],
\]
where \(B_{2\alpha}\) is the Bernoulli polynomial of degree \(2\alpha\).
Now, the CBC construction for a good generating vector $\bsz$ of a rank-1 lattice rule proceeds as in Algorithm~\ref{Algorithm_1}.

\begin{algorithm}[t]
\caption{CBC for $R_{N,\alpha,s,\bsgamma}$}
\label{Algorithm_1}
    Let \(d, N \in \mathbb{N}\), \(\alpha > \frac{1}{2}\) and \(\bsgamma\) be given.
\begin{enumerate}
    \item Let \(z_1^* = 1\) and \(\ell = 1\).
    \item Compute \(R_{N,\alpha,\ell+1,\bsgamma}(z_1^*, \dots, z_\ell^*, z_{\ell+1})\) for all \(z_{\ell+1} \in \{1, \dots, N - 1\}\) and let
    \[
    z_{\ell+1}^* = \arg \min_{z_{\ell+1}} R_{N,\alpha,\ell+1,\bsgamma}(z_1^*, \dots, z_\ell^*, z_{\ell+1}).
    \]
    \item If \(\ell + 1 < d\), let \(\ell = \ell + 1\) and go to Step 2.
\end{enumerate}
\end{algorithm}
The necessary computational cost for the CBC algorithm can be made small by using the fast Fourier transform \cite{nuyens2006fast}, that we only need $O(d N\log N) $ arithmetic operations with $O(N)$ memory.

Similarly to Algorithm \ref{Algorithm_1}, we can search for generating vector $\bsz$ by another CBC construction that just replaces the quantity $R_{N,\alpha,s,\bsgamma}$ by the quantity $S_{N,\alpha,s,\bsgamma}$.  As mentioned in \cite{cai2025l2approximationusingrandomizedlattice}, for a given generating vector $\bsz=(z_1,\ldots,z_s),$  when $\alpha>1/2$ is an integer,
$S_{N,\alpha,s,\bsgamma}$ has a computable formula as

\begin{align*}
 \left[S_{N,\alpha,s, \bsgamma}\left(\bsz_{s-1}, z_{s}\right)\right]^2 & =\frac{1}{N} \sum_{k=0}^{N-1} \theta_{\bsz_{s-1}, \alpha,\bsgamma}(k) \left(1+\gamma_s ^2 \frac{(-1)^{\alpha+1}(2\pi)^{2\alpha}}{(2\alpha)!} B_{2\alpha}\left(\left\{\frac{kz_s}{N}\right\}\right) \right)^2 \\
 & \quad -\prod_{j=1}^{s} (1+2\zeta(4\alpha)\gamma_j^4),
\end{align*}
where we write
\[ \theta_{\bsz_{s-1}, \alpha,\bsgamma}(k) = \prod_{j=1}^{s-1} \left(1+\gamma_j ^2 \frac{(-1)^{\alpha+1}(2\pi)^{2\alpha}}{(2\alpha)!} B_{2\alpha}\left(\left\{\frac{kz_j}{N}\right\}\right) \right)^2 . \]
To be precise, the corresponding CBC algorithm is described as in Algorithm~\ref{Algorithm_2}.

\begin{algorithm}[t]
\caption{CBC for $S_{N,\alpha,s,\bsgamma}$}
\label{Algorithm_2}
   Let \(d, N \in \mathbb{N}\), \(\alpha > \frac{1}{2}\) and \(\bsgamma\) be given.
\begin{enumerate}
    \item Let \(z_1^* = 1\) and \(\ell = 1\).
    \item Compute \(S_{N,\alpha,\ell+1,\bsgamma}(z_1^*, \dots, z_\ell^*, z_{\ell+1})\) for all \(z_{\ell+1} \in \{1, \dots, N - 1\}\) and let
    \[
    z_{\ell+1}^* = \arg \min_{z_{\ell+1}} S_{N,\alpha,\ell+1,\bsgamma}(z_1^*, \dots, z_\ell^*, z_{\ell+1}).
    \]
    \item If \(\ell + 1 < d\), let \(\ell = \ell + 1\) and go to Step 2.
\end{enumerate}
\end{algorithm}
The following bounds for  $R_{N,\alpha,s,\bsgamma} $ and $S_{N,\alpha,s,\bsgamma} $ hold.
\begin{proposition}\label{quantity_proposition}
Let $N$ be a prime number, let $d\in\mathbb{N},$ and let $ \bsgamma=(\gamma_j)_{j\geq 1}$ be product weights. Assume that $\bsz=(z_1,\ldots,z_d  )$ is a generate vector found by CBC construction for the weighted Korobov space. Then for arbitrary $\tau \in[1/2,\alpha)$ and for any $s\in \{ 1,2,\ldots,d\}$  we have
\begin{align}\label{CBCerror}
R_{N,\alpha,s,\bsgamma}(\bsz_{s-1},z_s)\leq \frac{2^\tau}{N^{\tau}}\prod_{j=1}^s \left(1+2\gamma_j^{1/(2\tau)}\zeta(\alpha/\tau) \right)^{\tau},
\end{align}
and
\begin{align}\label{CBCSerror}
S_{N,\alpha,s,\bsgamma}(\bsz_{s-1},z_s)\leq \frac{1}{N^\tau}\prod_{j=1}^s \left(1+2^{4\alpha+1} \zeta(\alpha/\tau)\gamma_j^{1/(2\tau)} \right)^{2\tau}.
\end{align}
\end{proposition}
\begin{proof}
The inequality \eqref{CBCerror} is shown in \cite{kuo2003component} and \eqref{CBCSerror} is in \cite{dick2007lattice}.
\end{proof}

\subsection{CBC Constructions for Polynomial Lattice Rules}
We now turn to the construction of good polynomial lattice rules in the weighted Walsh space $H_{d,\alpha,\bsgamma}^{\wal}$.
In this setting, we need to replace the definition of the hyperbolic cross $\mathcal{A}_{d}(M)$ with
\begin{align*}
\breve{\Acal}_{d}(M)=\{\bsh\in \NN_0^d:\left(\breve{r}_{\alpha,\bsgamma}(\bsh)  \right)^2 \leq M \}.
\end{align*}
The size of this new set $\breve{\Acal}_{d}(M)$ can be estimated as follows:

\begin{lemma}\label{lem:size_index_set_Walsh}
    For any $M\ge 1$ and $\lambda\in (1/(2\alpha),1]$, it holds that
    \[ \left| \breve{\Acal}_{d}(M)\right|\leq M^{\lambda}\prod_{j=1}^{d}\left(1+\gamma_{j}^2\frac{b-1}{b^{2\alpha\lambda}-b}\right).\]
\end{lemma}
\begin{proof}
We have
\begin{align*}
    \left| \breve{\Acal}_{d}(M)\right| & = \sum_{\bsh\in \NN_0^d}\II_{\left(\breve{r}_{\alpha,\bsgamma}(\bsh)  \right)^2 \leq M}\leq \sum_{\bsh\in \NN_0^d}\left( \frac{M}{\left(\breve{r}_{\alpha,\bsgamma}(\bsh)  \right)^2}\right)^{\lambda} \\
    & = M^{\lambda}\sum_{\bsh\in \NN_0^d}\frac{1}{\left(\breve{r}_{\alpha,\bsgamma}(\bsh)  \right)^{\lambda}}= M^{\lambda}\prod_{j=1}^{d}\left(1+\gamma_{j}^2\frac{b-1}{b^{2\alpha\lambda}-b}\right).
\end{align*}
This completes the proof.
\end{proof}
We also have a concept of dual polynomial lattice analogous to the dual lattice.
\begin{definition}[Dual polynomial lattice]
    Let $m\in \NN$, and let $p\in \FF_b[x]$ and $\bsq\in (\FF_b[x])^d$ be such that $\deg(p)=m$ and $\deg(q_j)<m$ for all $j$. Then, the set
    \[ P^{\perp}(p,\bsq):=\left\{\bsh \in \NN_0^{d} \mid \tr_m(\bsh(x)) \cdot \bsq(x) \equiv 0 \pmod {p(x)}\right\} \]
    is called the dual polynomial lattice of the polynomial lattice point set $P(p,\bsq)$, where $\tr_m(h(x)):=\eta_0+\eta_1x+\cdots+\eta_{m-1}x^{m-1}$ for $h=\eta_0+\eta_1b+\cdots$ and is applied component-wise to a vector.
\end{definition}
Then the following property of polynomial lattice rules holds.
\begin{lemma}[Character property]\label{lem:character_polynomial}
    Let $m\in \NN$, and let $p\in \FF_b[x]$ and $\bsq\in (\FF_b[x])^d$ be such that $\deg(p)=m$ and $\deg(q_j)<m$ for all $j$. Then,
    \[ \frac{1}{b^m}\sum_{\bsx\in P(p,\bsq)} \wal_{\bsh}(\bsx)= \begin{cases}1 & \text{if $\bsh \in P^{\perp}(p,\bsq)$,} \\ 0 & \text{otherwise,}\end{cases} \]
    holds for any vector $\bsh \in \NN_0^{d}$.
\end{lemma}
 With the help of character property, we can deduce the analytic form of integration error \eqref{int_error_criterion}, in a similar way to that done in the previous subsection. Denote $f\in  \mathbb{P}^{\wal}_I (\Omega)\setminus\{\boldsymbol{0}\}$ by $f(\bsx)=\sum_{\bsh\in I} a_{\bsh} \wal_{\bsh}(\bsx)$. Then we have 

\begin{align*}
\frac{1}{b^m}\sum_{\bsx\in P(p,\bsq)} f^2(\bsx)&=\frac{1}{b^m}\sum_{\bsx\in P(p,\bsq)}\left|\sum_{\bsh \in I} a_{\bsh}\wal_{\bsh}(\bsx)\right|^2\\
&=\sum_{\bsh'\in I}\sum_{\bsh''\in I} a_{\bsh'}\overline{a_{\bsh''}}\left(\frac{1}{b^m}\wal_{\bsh'\ominus \bsh''}(\bsx)\right)\\
&=\sum_{\bsh'\in I} a_{\bsh'}\sum_{\substack{\bsh''\in I\\ \bsh'\ominus \bsh''\in P^{\perp}(p,\bsq)}} \overline{a_{\bsh''}}
\end{align*}
and 
\begin{align}\label{walsh_integration_error}
\frac{1}{b^m}\sum_{\bsx\in P(p,\bsq)} f^2(\bsx)-\int_{\Omega}f^2(\bsx)\rd \bsx=\sum_{\bsh'\in I} a_{\bsh'}\sum_{\substack{\bsh''\in I\\ \bsh'\ominus \bsh''\in P^{\perp}(p,\bsq)\setminus \{\bszero\}}} \overline{a_{\bsh''}}. \end{align}

Analogously to Lemma~\ref{lem:assumption_error_bounds}, we obtain the following error bounds for weighted Walsh spaces.
\begin{lemma}\label{lem:polynomial_error_bounds}
Let $m\in \NN$, and let $p\in \FF_b[x]$ and $\bsq\in (\FF_b[x])^d$ be such that $\deg(p)=m$ and $\deg(q_j)<m$ for all $j$. The following error bounds hold
\begin{align}\label{polynomial lattice_integration_error}
\emph{err}_{b^m}(f,\mathbb{P}^{\wal}_I(\Omega))\leq \sqrt{|I|}(\| f\|^{\wal}_{d,\alpha,\bsgamma})^2 \breve{R}_{b^m,\alpha,d,\bsgamma}(p,\bsq),
\end{align}
where we write
\begin{align*}
\breve{R}^2_{b^m,\alpha,d,\bsgamma}(p,\bsq)=\sum_{\bsh \in P^{\perp}(p,\bsq)\setminus \{\bszero\}}\frac{1}{(\breve{r}_{\alpha,\bsgamma}(\bsh))^2}.
\end{align*}
\end{lemma}

\begin{proof}
 Since any polynomial $f\in \PP^{\wal}_{I}(\Omega)$ is of finite degree, $f\in H^{\wal}_{d,\alpha,\bsgamma}$ for any $\alpha>1/2$ and it follows from \eqref{walsh_integration_error} that
\begin{align*}
\left| \int_{\Omega} f^2(\bsx) \rd \bsx-\frac{1}{b^m}\sum_{\bsx\in P(p,\bsq)} f^2(\bsx)  \right|\leq \sum_{\bsh'\in I} |a_{\bsh'}|\sum_{\substack{\bsh''\in I\\ \bsh'\ominus \bsh''\in P^{\perp}(p,\bsq)\setminus \{\bszero\}}} |a_{\bsh''}|.
\end{align*}
Then, by using the Cauchy–-Schwarz inequality twice, we obtain
\begin{align*}
\text{err}_{b^m}(f,\mathbb{P}^{\wal}_I(\Omega))&\leq  \breve{R}_{b^m,\alpha,d,\bsgamma}(p,\bsq)\\
&\qquad\times\sum_{\bsh' \in I} |a_{\bsh'}|\left(\sum_{\substack{\bsh''\in I\\ \bsh'\ominus \bsh''\in P^{\perp}(p,\bsq)\setminus \{\bszero\}}} |a_{\bsh''}|^2 (\breve{r}_{\alpha,\bsgamma}(\bsh'\ominus \bsh''))^2\right)^{1/2} \\
&\leq \| p\|^{\wal}_{d,\alpha,\bsgamma} \breve{R}_{b^m,\alpha,d,\bsgamma}(p,\bsq)\\
&\qquad\times\sum_{\bsh'\in I}|a_{\bsh'}|\left( \sup_{\substack{\bsh''\in I\\ \bsh'\ominus \bsh''\in P^{\perp}(p,\bsq)\setminus \{\bszero\}}} \frac{(\breve{r}_{\alpha,\bsgamma}(\bsh'\ominus \bsh''))^2}{(\breve{r}_{\alpha,\bsgamma}(\bsh''))^2}\right)^{1/2} \\
&\leq (\| p\|^{\wal}_{d,\alpha,\bsgamma})^2 \breve{R}_{b^m,\alpha,d,\bsgamma}(p,\bsq)\\
&\qquad\times\left( \sum_{\bsh'\in I}\frac{1}{(\breve{r}_{\alpha,\bsgamma}(\bsh'))^2}\sup_{\substack{\bsh''\in \NN_0^d\\ \bsh'\ominus \bsh''\in P^{\perp}(p,\bsq)\setminus \{\bszero\}}} \frac{(\breve{r}_{\alpha,\bsgamma}(\bsh'\ominus \bsh''))^2}{(\breve{r}_{\alpha,\bsgamma}(\bsh''))^2}\right)^{1/2}.
\end{align*}

By using Lemma~\ref{lem:mu1_property}, for any $h,k\in \NN_0$, it holds that
\begin{align*}
    \breve{r}_{\alpha,\gamma}(h)  =\max\left\{ 1,\frac{b^{\alpha\mu_1(h)}}{\gamma}\right\} &\leq \max\left\{ 1,\frac{b^{\alpha\mu_1(k)}}{\gamma}, \frac{b^{\alpha\mu_1(h\ominus k)}}{\gamma}\right\}\\ &\leq \max\left\{ \breve{r}_{\alpha,\gamma}(k), \breve{r}_{\alpha,\gamma}(h\ominus k)\right\}\\
    & \leq \breve{r}_{\alpha,\gamma}(k) \breve{r}_{\alpha,\gamma}(h\ominus k).
\end{align*} 
It implies that, for any $\bsh',\bsh''\in \NN_0^d$, we have
\begin{align*}\frac{(\breve{r}_{\alpha,\bsgamma}(\bsh'\ominus \bsh''))^2}{(\breve{r}_{\alpha,\bsgamma}(\bsh''))^2}&=\frac{(\breve{r}_{\alpha,\bsgamma}(\bsh''\ominus \bsh'))^2}{(\breve{r}_{\alpha,\bsgamma}(\bsh''))^2}\\
&\leq \frac{(\breve{r}_{\alpha,\bsgamma}(\bsh'')\breve{r}_{\alpha,\bsgamma}((\bsh''\ominus \bsh')\ominus \bsh''))^2}{(\breve{r}_{\alpha,\bsgamma}(\bsh''))^2}=(\breve{r}_{\alpha,\bsgamma}(\bsh'))^2.
\end{align*}
Applying this inequality, we obtain the bound shown in \eqref{polynomial lattice_integration_error}. 
\end{proof}

If $I=\breve{\Acal}_d(M)$ for some $M\geq 1$, it follows from the definition of the weighted Walsh space that
\[ (\| p\|^{\wal}_{d,\alpha,\bsgamma})^2 = \sum_{\bsh\in I}|a_{\bsh}|^2(\breve{r}_{\alpha,\bsgamma}(\bsh))^2\leq M\sum_{\bsh\in I}|a_{\bsh}|^2\leq M\|p\|_{L_2}^2.\]
Moreover, for the error bound \eqref{polynomial lattice_integration_error}, the size of the index set is bounded above as shown in Lemma~\ref{lem:size_index_set_Walsh}. Thus, the remaining task is to find good $p$ and $\bsq$ such that $\breve{R}_{b^m,\alpha,d,\bsgamma}(p,\bsq)$ becomes small. Here, CBC construction can be used for the criteria as in Algorithm~\ref{Algorithm_4}.
Although we omit the detailed derivations, we have a computable formula for $\breve{R}_{b^m,\alpha,d,\bsgamma}(p,\bsq)$ as
\[  \breve{R}^2_{b^m,\alpha,d,\bsgamma}(p,\bsq)=-1+\frac{1}{b^m}\sum_{\bsx\in P(p,\bsq)}\prod_{j=1}^{d}\left( 1+\gamma_j^2 \phi_{\alpha}(x_j)\right),\]
where
\begin{align*}
    \phi_{\alpha}(x)
    &= \sum_{h=1}^{\infty} \frac{\wal_h(x)}{b^{2\alpha \mu_1(h)}}
     = \sum_{\ell=1}^{\infty} \frac{1}{b^{2\alpha \ell}} \sum_{h=b^{\ell-1}}^{b^{\ell}-1} \wal_h(x) \\
    &= 
    \begin{cases}
        \dfrac{b-1}{b^{2\alpha} - b}, & \text{if } x = 0, \\[2ex]
        \dfrac{b-1}{b^{2\alpha} - b}
        - \dfrac{1}{b^{2\alpha c_0}} \cdot \dfrac{b^{2\alpha} - 1}{b^{2\alpha} - b},
        & \begin{aligned}
            &\text{if } x = \dfrac{x_1}{b} + \dfrac{x_2}{b^2} + \cdots, \\
            & \text{ with } x_1 = \cdots = x_{c_0-1} = 0\text{ and } x_{c_0} \neq 0.
        \end{aligned}
    \end{cases}
\end{align*}

\begin{algorithm}[t]
\caption{CBC for polynomial lattice rules}
\label{Algorithm_4}
    Let \(d, m \in \mathbb{N}\), \(\alpha > \frac{1}{2}\) and \(\bsgamma\) be given. Let $p\in \FF_b[x]$ be irreducible with $\deg(p)=m$.
\begin{enumerate}
    \item Let \(q_1^* = 1\in \FF_b[x]\) and \(\ell = 1\).
    \item Compute \(\breve{R}_{b^m,\alpha,\ell+1,\bsgamma}(p,(q_1^*, \dots, q_\ell^*, q_{\ell+1}))\) for all $q_{\ell+1}\in\FF_b[x]$ such that $\deg(q_{\ell+1})<m$ and let
    \[
    q_{\ell+1}^* = \arg \min_{z_{\ell+1}} \breve{R}_{b^m,\alpha,\ell+1,\bsgamma}(p,(q_1^*, \dots, q_\ell^*, q_{\ell+1})).
    \]
    \item If \(\ell + 1 < d\), let \(\ell = \ell + 1\) and go to Step 2.
\end{enumerate}
\end{algorithm}

The following bound for $\breve{R}_{b^m,\alpha,d,\bsgamma}(p,\bsq)$ holds.
\begin{proposition}\label{polynomial_quantity_proposition}
For $m\in \NN$, let $p\in \FF_b[x]$ be irreducible with $\deg(p)=m$, let $d\in\mathbb{N},$ and let $ \bsgamma=(\gamma_j)_{j\geq 1}$ be product weights. Assume that $\bsq=(q_1,\ldots,q_d)\in (\FF_b[x])^d$ is a generating vector constructed by the CBC algorithm for the weighted Walsh space. Then, for arbitrary $\tau \in[1/(2\alpha),1)$ and for any $s\in \{ 1,2,\ldots,d\}$, we have
\begin{align*}
\breve{R}_{b^m,\alpha,s,\bsgamma}(p,(q_1,\ldots,q_s))\leq \left( \frac{2}{b^m-1}\left(-1+\prod_{j=1}^{s}\left(1+\gamma_{j}^2\frac{b-1}{b^{2\alpha\lambda}-b}\right)\right)\right)^{1/(2\lambda)}.
\end{align*}
\end{proposition}

 Similarly to Lemma~\ref{lem:assumption_error_bounds}, we can derive another bound on $\text{err}_{b^m}(f,\mathbb{P}^{\wal}_I(\Omega))$ with a different computable quality criterion in this setting. However, showing an upper bound on that criterion is not straightforward, so that we do not discuss that bound in this work. Investigating another bound remains a topic for future research.

\subsection{Escaping from the curse of dimensionality}
QMC rules are known to often overcome the curse of dimensionality in certain high-dimensional integration problems. Since QMC hyperinterpolation is constructed by QMC rules, it naturally inherits this favorable property and can likewise escape from the curse of dimensionality in specific $L_2$ approximation problems for smooth functions. Let us consider the following worst-case error for the $L_2$ approximation problem in the weighted Korobov space $H_{d,\alpha,\bsgamma},$ defined as
\begin{align}\label{L_2-worst-error_def}
\text{err}^{L_2\text{-app}}(H_{d,\alpha,\bsgamma}, A_{N,d} (f)) := \sup_{\substack{f \in H_{d,\alpha,\bsgamma} \\ \|f\|_{d,\alpha,\bsgamma} \leq 1}} \left\| A_{N,d}(f)-f\right\|_{L_2},
\end{align}
where $A_{N,d}(f)$ is defined as the general form of standard linear approximation algorithms, that is 
\begin{align*}
A_{N,d}(f) := \sum_{n=0}^{N-1} a_n f(\bsx_n), \quad a_n \in L_2(\Omega).
\end{align*}
We denote all $ A_{N,d}(f)$ by $\Lambda,$ and it is clear that the QMC hyperinterpolations $\mathcal{Q}_I f$ constructed from Assumption \ref{Assum:1} are included in $  A_{N,d}(f).$ Then by adapting the definition of information complexity to \eqref{L_2-worst-error_def}, we can illustrate the phenomenon of \emph{the curse of dimensionality} for approximation problems.

\begin{definition}[Information complexity of approximation problem]
For given $\epsilon\in(0,1)$ and $d\in\mathbb{N},$ and given $N$ distinct function evaluations $f(\bsx_n)$ for some $\bsx_n\in \Omega,\;n=0,1,\ldots,N-1.$ The information complexity for the  $L_2$-approximation problem in the weighted Korobov space $H_{d,\alpha,\bsgamma}$ is defined by 
\begin{align*}
N(\epsilon,d):=\min\left\{ N\in\mathbb{N}:  \exists A_{N,d} (f)\in \Lambda\;\text{such that} \;\mathrm{err}^{L_2\text{-}\mathrm{app}}
(H_{d,\alpha,\bsgamma}, A_{N,d})\leq \epsilon \right\}.
\end{align*}
\end{definition}
\begin{definition}[Curse of dimensionality]\label{def:curse}The $L_2$ approximation problem \eqref{L_2-worst-error_def} is said to suffer from the \emph{curse of dimensionality}, if there exists numbers $C_1>0,\;b>0,\;$ and $\epsilon_0\in(0,1)$ such that
\begin{align*}
N(\epsilon,d)\geq C_1(1+b)^d\;\text{for\;all}\;\epsilon\in(0,\epsilon_0)\;\text{and\;for\;infinitely\;many}\;d\in\mathbb{N}.
\end{align*}
\end{definition}
 
The approximation problems, which do not suffer from the curse of dimensionality, are usually classified by various notions of tractability.  In particular, the $L_2$ approximation problem \eqref{L_2-worst-error_def} is said to be \emph{polynomially tractable}, if there exists constants $C_2,\sigma>0$ and $t\geq0$ such that
\begin{align}\label{bound_information_complexity}
N(\epsilon,d)\leq C_2d^{t}\epsilon^{-\sigma}\;\text{for\;all}\;\epsilon\in(0,1)\;\text{and\;all} \;d\in\mathbb{N}.
\end{align}
Obviously, the polynomial tractability means escaping from the curse of dimensionality, as there does not exist any $b>0$ such that Definition~\ref{def:curse} applies. Moreover, the $L_2$ approximation problem \eqref{L_2-worst-error_def} is said to be \emph{strongly polynomially tractable} if $t=0$ holds for \eqref{bound_information_complexity} and the infimum of those $\sigma$ for which \eqref{bound_information_complexity} satisfies with $t=0$ is called the $\epsilon$-exponent of strong polynomial tractability for the class $\Lambda.$

The polynomial tractability and strong polynomial tractability of the $L_2$ approximation problem \eqref{L_2-worst-error_def} are general results from \cite{Novak2004}. Since the QMC hyperinterpolation $\mathcal{Q}_If$ based on rank-1 lattice rules is a special case of $A_{N,d}(f)$, we can state the tractability of $\mathcal{Q}_I f$  for \eqref{L_2-worst-error_def} by the following theorem, which is from \cite[Theorem 13.11]{dick2022lattice}.

\begin{theorem}\label{lattice_tractability}
The QMC hyperinterpolation with rank-1 lattice rules for the $L_2$ approximation problem in the weighted Korobov space $H_{d,\alpha,\bsgamma}$ exhibits the property of strong polynomial tractability if 
\begin{align}\label{tractable_1}
\sum_{j=1}^{\infty} \gamma_{j}^{1/(2\tau)}<\infty
\end{align}
holds for some $\tau\in[1/2,\alpha),\;\alpha>1/2.$ Here, the corresponding $\epsilon$-$exponent$ is at most $2/\tau$.  Furthermore, polynomial tractability holds if
\begin{align}\label{tractable_2}
\limsup\limits_{d \rightarrow \infty}\frac{1}{\log d}\sum_{j=1}^d \gamma_j<\infty.
\end{align}
\end{theorem}

\begin{remark}
    Although we omit the details, similar results hold for the weighted Walsh space $H^{\wal}_{d,\alpha,\bsgamma}$. That is, under the same summability condition \eqref{tractable_1}, the $L_2$ approximation problem is strongly polynomially tractable, whereas the corresponding $\epsilon$-exponent becomes larger. Moreover, the condition \eqref{tractable_2} ensures the polynomial tractability as well.
\end{remark}

\section{Hyperinterpolation versus QMC hyperinterpolation}\label{sec:comparison}
In this section, we compare the computation costs and the approximation accuracy of hyperinterpolation and QMC hyperinterpolation based on rank-1 lattice point sets. First of all, in most cases, the computation costs of constructing quadrature points for hyperinterpolation are much higher than QMC hyperinterpolation, because of their different assumptions about the quadrature rules. We illustrate this fact by the QMC rules satisfying the exactness assumption \eqref{hyper_exact_assump} for hyperinterpolation over $\mathbb{P}_I(\Omega)$. Such QMC rules also represent a special case of QMC hyperinterpolation. From the integration error \eqref{integration_error}, it is easy to verify that
\begin{align*}
\int_{\Omega} p^2(\bsx) \rd \bsx=\frac{1}{N}\sum_{n=0}^{N-1}p^2(\bsx_n),\;\;\forall p\in \mathbb{P}_{I}(\Omega),\;\;\bsx_n\in P_{N,\bsz}
\end{align*}
if and only if
\begin{align} \label{reconstruct}
\bsh\cdot \bsz \not\equiv \bsh'\cdot \bsz \pmod{N},\;\;\text{for}\;\text{all}\;\bsh,\bsh'\in I,\;\bsh\neq \bsh'.
\end{align}
We can refer to the equality \eqref{reconstruct} as the reconstruction property. Then, we have the following hyperinterpolation
\begin{align*}
\mathcal{L}_I f=\sum_{\bsh\in I}\left(\frac{1}{N}\sum_{n=0}^{N-1} f\left(\bsx_n\right)\exp\left(-2\pi i \bsh\cdot \bsx_n\right)\right)\exp(2\pi i\bsh \cdot \bsx),
\end{align*}
with rank-1 lattice point sets $\{\bsx_0,\ldots,\bsx_{N-1}\}\in P_{N,\bsz}$, where the generating vector $\bsz$ satisfies the reconstruction property \eqref{reconstruct}.  Hence, the selection criteria of CBC algorithm \ref{cbc_principle} for $\bsz$ in \eqref{reconstruct} would be the solution of the congruence equations. The special CBC construction, as described in Algorithm~\ref{Algorithm_3}, is also found in \cite[Algorithm 1]{Kammerer_Reconstruction_property_2014}.  
\begin{algorithm}[t]
\caption{CBC for reconstruction property}
\label{Algorithm_3}
    Let \(d, N \in \mathbb{N}\), and index set \(I\) be given.
\begin{enumerate}
    \item Let \(z_1^* = 1\), \(\ell = 1\), and 
    \[
    I_\ell := \left\{ (h_j)_{j=1}^\ell:\bsh = (h_j)_{j=1}^d \in I \right\}.
    \]
    
    \item Search for one \(z_{\ell+1} \in \{1, \dots, N - 1\}\) with
    \begin{align*}
    \left| \left\{
    (z_1^*, \dots, z_\ell^*, z_{\ell+1}) \cdot \bsk \bmod N 
    : \bsk \in I_\ell 
    \right\} \right| = |I_\ell|,
    \end{align*}
    and let
    \begin{align*}
    z_{\ell+1}^* = (z_1^*, \dots, z_\ell^*, z_{\ell+1}).
    \end{align*}

    \item If \(\ell + 1 < s\), let \(\ell = \ell + 1\) and go to Step 2.
\end{enumerate}

\end{algorithm}

Unlike the selection criteria in Algorithm \ref{Algorithm_1} and \ref{Algorithm_2}, which minimizes certain quality measures such as \( R_{N,\alpha,s,\bsgamma} \) and \( S_{N,\alpha,s,\bsgamma} \), it is difficult to apply the selection criterion in Algorithm \ref{Algorithm_3}. As noted in~\cite{Kammerer_Volkmer_2019}, the algorithm for constructing a rank-1 lattice rule that satisfies the reconstruction property requires \(\mathcal{O}(|I|^3 + d|I|^2 \log |I|)\) arithmetic operations, which is significantly slower than Algorithm~\ref{Algorithm_1} and Algorithm~\ref{Algorithm_2}, both of which require only \(\mathcal{O}(dN \log N)\) operations. 

For the approximation accuracy, we compare the general error bounds of hyperinterpolation~\eqref{hyper_L_2_error} and QMC hyperinterpolation~\eqref{equ:L_2error}. Although the latter includes an additional aliasing term $\Vert \mathcal{Q}_I p^* - p^* \Vert_{L_2}$, this does not necessarily imply a slower convergence of $\mathcal{Q}_I f$
 than $\mathcal{L}_I f.$ The reason is that, within the same number of rank-1 lattice points, the fixed relationship between the index set $I$ and the number of quadrature points $N$ given in~\eqref{reconstruct} might lead to a smaller range $I$ of $E_I(f)$ for $\mathcal{L}_I f.$ From \cite[Theorem 13.6]{dick2022lattice}, for $f\in H_{d,\alpha,\bsgamma}$ and a weighted hyperbolic cross index $I_{N,\tau}:=\mathcal{A}_d(N^{\tau}),\;\tau\in[1/2,\alpha),$ the QMC hyperinterpolation with the generating vector $\bsz$ from Algorithm \ref{Algorithm_2} follows the worst-case $L_2$ error bound
 \begin{align*}
\sup_{\substack{f\in H_{d,\alpha,\bsgamma}\\ \Vert f\Vert_{d,\alpha,\bsgamma}\leq  1}}\Vert f- \mathcal{Q}_{I_{N,\tau}}(f)\Vert_{L_2} \leq \frac{C_{d,\alpha,\bsgamma,\tau}}{N^{\tau/2}},
 \end{align*}
 where $C_{d,\alpha,\bsgamma,\tau}$ is bounded uniformly in $d$ if $\sum_{j=1}^\infty \gamma_j^{1/(2\tau)}<\infty,$ which verifies the strong polynomial tractability of QMC hyperinterpolation in Theorem \ref{lattice_tractability}. However, the hyperinterpolation with rank-1 lattice rules can not improve upon the convergence rate of order $N^{-\alpha/2}$ with any index set; this fact follows from the lower bound of \eqref{L_2-worst-error_def} for rank-1 lattice points set in \cite{Byrenheid2017}.
\begin{theorem}\label{thm:hyper_by_multiple_lattice}
Let $d,\;N\in \NN,$ let $\alpha>1/2,$ and let $\bsgamma=(\gamma_j)_{j\geq 1}$ be product weights. Furthermore, let $A_{N,d}(\bsz)$ be an arbitrary linear $L_2$ approximation algorithm using function evaluations at the points of a rank-1 lattice point set with generating vector $\bsz$ for $f\in H_{d,\alpha,\bsgamma}.$ Then it is true that
\begin{align*}
 \sup_{\substack{f\in H_{d,\alpha,\bsgamma}\\ \Vert f\Vert_{d,\alpha,\bsgamma}\leq  1}}\Vert f- A_{N,d}(\bsz)(f)\Vert_{L_2} \geq \frac{C_{d,\alpha,\bsgamma}}{N^{\alpha/2}},
\end{align*}
where $C_{d,\alpha,\bsgamma}$ is a positive real that is independent of $N.$
\end{theorem}

\begin{remark}
To design an approximation scheme with higher approximation accuracy and low computational cost, K\"ammerer proposed the use of a union of rank-1 lattice rules, see \cite{Kammerer_2019_multilattice}. For $N=N_1+\ldots+N_L,$ it was shown in \cite{Kammerer_Volkmer_2019} that the worst-case \(L_2\) approximation error \eqref{L_2-worst-error_def} of the algorithm \(A_{N,d}(\bsz_1,\ldots,\bsz_L)(f)\) decays at a rate of order \(\mathcal{O}\left(N^{-\alpha + 1/2 + \epsilon}\right)\) for $\alpha>1/2$ when using a weighted hyperbolic cross index set, where \(\epsilon > 0\) can be arbitrarily small. This result improves upon the convergence rate of order $N^{-\alpha/2}$ when $\alpha>1.$ However, we emphasize that the multiple rank-1 lattice rules are not classical quadrature rules, as their quadrature weights vary with the choice of index set $I.$   
\end{remark}

\begin{remark}
For the exactness assumption of hyperinterpolation over $\mathbb{P}_I^{\wal}(\Omega),$ there exists a similar equality to the reconstruction property \eqref{reconstruct}. From the integration error \eqref{walsh_integration_error}, we have $\text{err}_{b^m}(f,\mathbb{P}_I^{\wal}(\Omega))=0$ if and only if $\tr_m(\bsh(x)) \cdot \bsq(x) \not\equiv \tr_m(\bsh'(x)) \cdot \bsq(x) \pmod {p(x)}$ for all $\bsh,\bsh'\in I$ and $\bsh \neq \bsh'.$ As we know from the present literature, there are no algorithms to realize such a similar reconstruction property, but we believe it is also numerically difficult to design.
\end{remark}

\section{Lasso QMC hyperinterpolation}\label{sec:lasso}
 The QMC hyperinterpolation is good at dealing with high-dimensional approximation problems about smooth functions. However, when we construct $\mathcal{Q}_I f$ using discrete function values, it is often the case that we cannot avoid noise in the sampling process, such as the observing error and the rounding error. To extend QMC hyperinterpolation to more realistic problems, we combine Lasso (``Least Absolute Shrinkage and Selection Operator'') with $\mathcal{Q}_I f$ to reduce noise. We first decompose the function approximation problem $\Vert \mathcal{Q}_If^{\boldsymbol{\epsilon}}-f\Vert_{L_2}$ into an undetermined polynomial approximation problem $\Vert p-\mathcal{Q}_I f\Vert_{L_2}.$ Then discretize the \(L_2\) norm using quadrature points that satisfy the exactness condition~\eqref{hyper_exact_assump} for all \(p \in \mathbb{P}_I(\Omega)\), and introduce an additional \(\ell_1\)-regularization term. Note that the quadrature points used to discretize the \(L_2\) norm are different from those used in the construction of \(\mathcal{Q}_I f\). This leads to a new approximation scheme, which is analyzed using statistical tools. Specifically, let $\boldsymbol{\epsilon} \in \mathbb{R}^{N}$ have independent components $\epsilon_j \sim \mathcal{N}(0, \sigma^2).$ Consider the $L_2$ approximation of $f\in L_2(\Omega)$ by $\mathcal{Q}_If$ with $N$ noisy data $f^\epsilon(\bsx_n)=f(\bsx_n)+\epsilon_n$ over $N$ quadrature points $\bsx_n,\;n=0,\ldots,N-1.$  Then 
 \begin{align*}
\Vert \mathcal{Q}_I f^\epsilon-f\Vert_{L_2}\leq \Vert \mathcal{Q}_I f^\epsilon-\mathcal{Q}_I f\Vert_{L_2}+\Vert \mathcal{Q}_I f-f\Vert_{L_2}.
 \end{align*}
 Since 
 \begin{align*}
\mathcal{Q}_I f^{\epsilon}=\sum_{\bsh\in I}\left( \frac{1}{N}\sum_{n=0} ^{N-1} f(\bsx_n)\overline{q_{\bsh}(\bsx_n) }+\frac{1}{N}\sum_{n=0} ^{N-1} \epsilon_n\overline{q_{\bsh}(\bsx_n) }\right)q_{\bsh}(\bsx)
 \end{align*}
 and
 \begin{align*}
\Vert \mathcal{Q}_I f^\epsilon-\mathcal{Q}_I f\Vert_{L_2}^2=\int_{\Omega}\left(\frac{1}{N} \sum_{n=0}^{N-1}\epsilon_n \left(\sum_{\bsh\in I}q_{\bsh}(\bsx)\overline{q_{\bsh}}(\bsx_n) \right) \right)^2 \rd \bsx,
 \end{align*}
with $\mathbb{E}(\epsilon_n)=0,\;n=0,\ldots,N-1,$  by considering the expectation of  $\Vert \mathcal{Q}_I f^\epsilon-\mathcal{Q}_I f\Vert_{L_2}^2$ for $\epsilon,$ we have
\begin{align*}
\mathbb{E} \left(\Vert \mathcal{Q}_I f^\epsilon-\mathcal{Q}_If\Vert_{L_2}^2\right)&=\frac{\sigma^2}{N^2}\sum_{n=0}^{N-1}\int_{\Omega} \left(\sum_{\bsh\in I}q_{\bsh}(\bsx)\overline{q_{\bsh}}(\bsx_n) \right)^2 \rd \bsx \\
&=\frac{\sigma^2}{N^2}\sum_{n=0}^{N-1}\sum_{\bsh\in I} q_{\bsh}(\bsx_n)\overline{q_{\bsh}(\bsx_n)} =\frac{\sigma^2|I|}{N},
\end{align*}
where the second equality follows from the orthogonality of $q_{\bsh}$. Then, by the Cauchy–-Schwarz inequality, we have 
\begin{align}\label{expection_noiseerror}
\mathbb{E} \left(\Vert \mathcal{Q}_I f^\epsilon-\mathcal{Q}_If\Vert_{L_2}\right)\leq \sqrt{\mathbb{E} \left(\Vert \mathcal{Q}_I f^\epsilon-\mathcal{Q}_If\Vert_{L_2}^2\right)}= \sigma\sqrt{\frac{|I|}{N}}  .
 \end{align}
The upper bound \eqref{expection_noiseerror} will increase with $|I|,$, making it difficult to get a reasonable $L_2$ approximation error bound by QMC hyperinterpolation with noisy data.  To obtain a more stable approximation scheme for $f\in L_2(\Omega)$ with noisy data, we wish to design a new approximation scheme related to QMC hyperinterpolation based on regularization techniques.

For short, let us write $\mathcal{Q}_I f:=\sum_{\bsh\in I} \tilde{\beta}_{\bsh} q_{\bsh}(\bsx),$ and assume that a new approximation scheme is given by $p=\sum_{\bsh\in I}\beta_{\bsh}q_{\bsh}(\bsx)$ with undetermined coefficients $\beta_{\bsh}.$ 
Denoting $\boldsymbol{\beta}=(\beta_{\bsh})_{\bsh\in I} $ and $\tilde{\boldsymbol{\beta}}=(\tilde{\beta}_{\bsh})_{\bsh\in I}, $ given an $\ell_1$ regularization penalty to $p$, we have the following $\ell_1$-regularization least squares problem. 
\begin{align}\label{l1_model}
\min_{p\in \mathbb{P}_I(\Omega)}\left\{\Vert p-\mathcal{Q}_I f\Vert_{L_2}^2+\lambda\Vert \boldsymbol{\beta} \Vert_1\right\},\;\;\;\;\lambda>0.
\end{align}
It should be noted that the integration for $L_2$ error in \eqref{l1_model} can be accurately evaluated using the QMC rules that satisfy the exactness assumption \eqref{hyper_exact_assump}, and then we make a new data sampling (with noise $\epsilon$) for both $p$ and $\mathcal{Q}_I f$ over quadrature points of QMC rules for \eqref{hyper_exact_assump} and the model \eqref{l1_model} can be rewritten into the matrix form
\begin{align}\label{matrix_ell_1}
\hat{\boldsymbol{\beta}}: =\arg\min_{\boldsymbol{\beta}}\left\{\frac{1}{N}\Vert \mathbf{y}-\mathbf{X}\boldsymbol{\beta} \Vert_2^2+\lambda\Vert \boldsymbol{\beta} \Vert_1\right\},\;\;\;\;\lambda>0,
\end{align}
where $\mathbf{X} = \left( q_{\bsh}(\bsx_n') \right)_{\bsh \in I,\; n = 0,\ldots,N-1}$ is the Fourier matrix evaluated at quadrature points of QMC rules satisfy \eqref{hyper_exact_assump}, and $ \mathbf{y}$ is the newly function values (with noisy $\epsilon$) of $\mathcal{Q}_I f$ over $\bsx_n',$ that is
\begin{align*}
\mathbf{y} = \mathbf{X}\tilde{\boldsymbol{\beta}} + \boldsymbol{\epsilon} \in \mathbb{R}^{N}.
\end{align*}
Then we have the following theorem about the solution to \eqref{matrix_ell_1}.

\begin{theorem}\label{matrix_X*X=I}
For a given finite index set $I$ with finite cardinality $|I|,$ and QMC rules that  satisfy the exactness assumption \eqref{hyper_exact_assump}, the unique solution of Lasso \eqref{matrix_ell_1} is 
\begin{align}\label{lasso_analytic_solution}
\hat{\boldsymbol{\beta}}=\eta_S\left(\frac{1}{N}  \mathbf{X}^* \mathbf{y},\lambda\right),
\end{align}
where $\eta_S(a,k)$ is a \emph{soft thresholding operator} \cite{Donoho_1994}
\begin{equation*}
 \eta_{S}(a,k):=\max(0,a-k)+\min(0,a+k).
\end{equation*}
\end{theorem}
\begin{proof}
Note that any two distinct columns of the Fourier matrix $\mathbf{X}$ are either orthogonal or equal to each other. This follows because
\begin{align*}
\left(  \mathbf{X^* X}\right)_{\bsh',\;\bsh}=\sum_{n=0}^{N-1} q_{\bsh}(\bsx_n')\overline{q_{\bsh'}(\bsx_n')}=\begin{cases}
N\;\; if \;\bsh=\bsh'\\
0\;\; else.
\end{cases}
\end{align*}
Then by \cite[Theorem 3.4]{Lasso_hyper_2021}, we have \eqref{lasso_analytic_solution}.
\end{proof}
By denoting  
\[
\boldsymbol{\epsilon}' = \frac{1}{N} \mathbf{X}^* \boldsymbol{\epsilon} \quad \text{and} \quad \boldsymbol{\epsilon}'' = \left(\frac{1}{N} \sum_{n=0}^{N-1} \epsilon_n \overline{q_{\bsh}(\bsx_n)}\right)_{\bsh\in I},\;
\]  
it is clear that both \(\boldsymbol{\epsilon}'\) and \(\boldsymbol{\epsilon}''\) follow the distribution \(\mathcal{N}(0, \sigma^2/N)\). Therefore, the analytical solution \(\hat{\boldsymbol{\beta}}\) coincides with the coefficients obtained by applying a soft-thresholding operator to \(\mathcal{Q}_I f^{\epsilon}\), since  
\[
\frac{1}{N} \mathbf{X}^* \mathbf{y} = \frac{1}{N} \mathbf{X}^*\mathbf{X}\tilde{\boldsymbol{\beta}}+\frac{1}{N}\mathbf{X}^*\boldsymbol{\epsilon}=\tilde{\boldsymbol{\beta}} + \boldsymbol{\epsilon}'
\]  
and the coefficients of \(\mathcal{Q}_I f^{\boldsymbol{\epsilon}}\) correspond to \(\tilde{\boldsymbol{\beta}} + \boldsymbol{\epsilon}''\).

We directly define this new approximation scheme by Lasso QMC hyperinterpolation, which is similar to Lasso hyperinterpolation proposed in \cite{Lasso_hyper_2021}. 
 \begin{definition}[Lasso QMC hyperinterpolation]
     Given a real-valued function $f\in L_2(\Omega),$ and an index set $I$ with finite cardinality $|I|$, assume that a QMC rule with $N$ quadrature points satisfies Assumption \ref{Assum:1}. Then the Lasso QMC hyperinterpolation is defined by
     \begin{align*}
\mathcal{Q}^{\lambda}_I f:=\sum_{\bsh\in I}\eta_{S}\left(\left< f,q_{\bsh}\right>_N,\lambda\right) q_{\bsh}.
\end{align*}
 \end{definition}
The quadrature points for Lasso QMC hyperinterpolation are not required to satisfy the exactness assumption, which is the main difference from Lasso hyperinterpolation. The following compatibility condition and the consistency theorem can be found in \cite[Chapter 6]{Peter_Lasso_book_2011}.
\begin{definition}[Compatibility condition]
Assume that \( S_0 \subset \{1, \ldots, |I| \} \), and define
\[
\boldsymbol{\beta}_{S_0} := \boldsymbol{\beta} \cdot \mathbf{1}_{S_0},
\]
where \( \mathbf{1}_{S_0} \in \mathbb{R}^{|I|} \) is the indicator (or characteristic) function of the set \( S_0 \). For the regularization model \eqref{matrix_ell_1}, we say that the compatibility condition is met for the set $S_0$, if for some $\phi_0> 0,$  and for all $\boldsymbol{\beta}$ satisfying $\Vert \boldsymbol{\beta}_{S_0^c} \Vert_1\leq 3\Vert\boldsymbol{\beta}_{S_0} \Vert_1,$ it holds that
\begin{align*}
\Vert \boldsymbol{\beta}_{S_0} \Vert_1^2\leq \left(\boldsymbol{\beta}^* \hat{\Sigma} \boldsymbol{\beta}\right) s_0/\phi_0^2,
\end{align*}
where $ \hat{\Sigma}=\frac{1}{N}\mathbf{X}^*\mathbf{X}$ and $s_0=|S_0|.$
\end{definition}
\begin{theorem}[Consistency of the Lasso]
For the regularization model \eqref{matrix_ell_1}, assume that the compatibility condition
holds for $S_0$. For some $t>0,$ let the regularization parameter be
\begin{align*}
\lambda=4\hat{\sigma}\sqrt{ \frac{t^2+2\log |I| }{N}},
\end{align*}
where $ \hat{\sigma}$ is some estimator of $\sigma.$  Then with probability at least $1-\mu,$ where 
\begin{align*}
\mu:=2\exp(-t^2/2)+\mathbf{P}(\hat{\sigma}\leq \sigma),
\end{align*}
we have 
\begin{align*}
\Vert \mathbf{X}(\hat{\boldsymbol{\beta}}-\tilde{\boldsymbol{\beta}})\Vert_2^2/N+\lambda\Vert \hat{\boldsymbol{\beta}}-\tilde{\boldsymbol{\beta}}\Vert_1\leq 4\lambda^2 s_0 /\phi_0^2,
\end{align*}
with $\hat{\boldsymbol{\beta}},\;\tilde{\boldsymbol{\beta}} $ corresponding to the coefficients of $\mathcal{Q}_I f$ and $\mathcal{Q}_I^\lambda f^{\epsilon}$ respectively.
\end{theorem}
With the consistency of the Lasso and the exactness assumption of quadrature rules, we can deduce the following result.
\begin{corollary}
\begin{align}\label{lasso_denoise_error}
\Vert \mathbf{X}(\hat{\boldsymbol{\beta}}-\tilde{\boldsymbol{\beta}})\Vert_2^2/N=\Vert \mathcal{Q}_I ^{\lambda} f^{\epsilon}-\mathcal{Q}_I f\Vert_{L_2} \leq \frac{2\lambda\sqrt{s_0}  }{\phi_0},
\end{align}
with a larger probability for the proper choice of $\lambda.$
The upper bound \eqref{lasso_denoise_error} shows the great denoising ability of Lasso, since it reduces the term $|I|$ in \eqref{expection_noiseerror} by $\log |I|.$ 
\end{corollary}

\section{Numerical experiments}\label{Numerical_experiments}
In this section, we show some numerical experiments to verify our theoretical findings. We first compare the computational costs of hyperinterpolation $\mathcal{L}_I f$ and QMC hyperinterpolation $\mathcal{Q}_If $ using the same index set $I$. Our comparison is restricted to rank-1 lattice rules for both methods. We set the index set $I$ by the weighted hyperbolic cross set 
\begin{align*}
\mathcal{A}_{d,M,\tau,\bsgamma}=\left\{\bsh\in \mathbb{Z}^d:(r_{\alpha,\bsgamma}(\bsh))^2
\leq M^\tau \right\},
\end{align*}
and we measure the computation time with different size $M.$ The generating vector $\bsz$ of a rank-1 lattice rule for $\mathcal{L}_I f$ should satisfy the reconstruction property \eqref{reconstruct}, which can be CBC constructed by the implementation \cite[\texttt{r1l\_cbc\_search}]{kaemmererLFFT}. For the CBC construction of the generating vector $\bsz$ of a rank-1 lattice rule of $\mathcal{Q}_If,$  since the number of quadrature points for $\mathcal{Q}_I f$ does not rely on the index set $I,$ we construct it by the Algorithm \ref{Algorithm_2} with the prime number that is closest to the number of quadrature points of $\mathcal{L}_I f.$ Figure \ref{Figure_time} shows the lower computation costs of QMC hyperinterpolation comparing to hyperinterpolation.

\begin{figure}[t]
  \centering
  \includegraphics[width=0.8\linewidth]{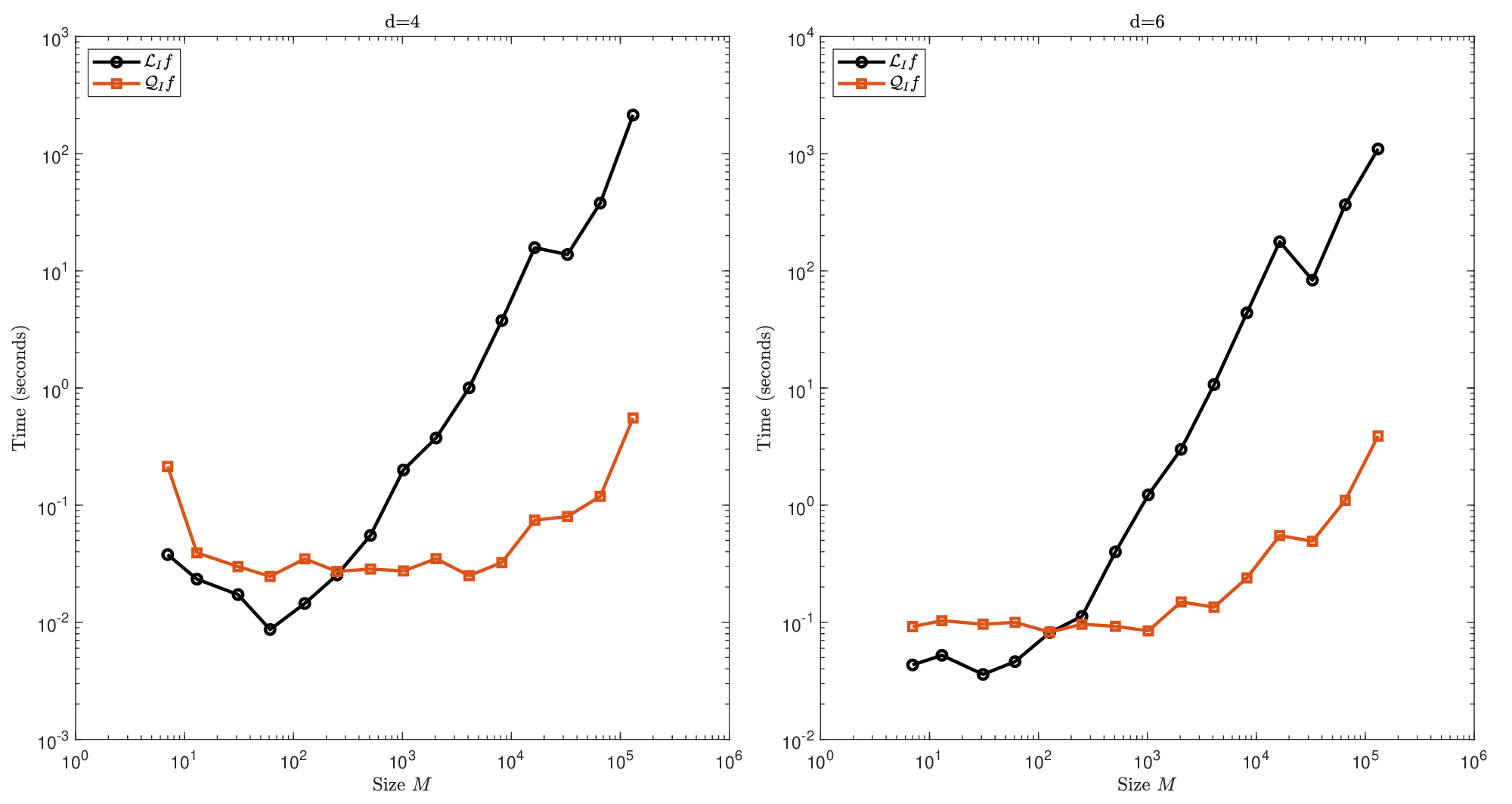}\\
\caption{Computation time for $d$ dimensional hyperinterpolation and QMC hyperinterpolation with parameters $\alpha=2,\;\tau=1.9,\;\gamma_j=1/j^2,\;j=1,\ldots,d.$} \label{Figure_time}
\end{figure}

Then, we investigate the approximation efficiency of QMC hyperinterpolation for the following test function:
\begin{align*}
f(\bsx )=\prod_{j=1}^d \frac{8\sqrt{6}\sqrt{\pi}}{\sqrt{6369\pi-4096}}\left[4+{\rm sgn}\left(x_j-\frac{1}{2}\right)\left(     \sin^3 (2\pi x_j)  +\sin^4(2\pi x_j)\right)\right].
\end{align*}
This function \( f \in H_{d, \frac{7}{2}, \boldsymbol{\gamma}} \) is taken from~\cite{Kammerer_Volkmer_2019}. The Fourier coefficients of $f$ can be computed analytically; see \cite[Section 5]{Kammerer_Volkmer_2019}. Thus, we evaluate the $L_2$ approximation error of $f$ by
\begin{align*}
\Vert f-\mathcal{Q}_I f\Vert_{L_2}^2 =\sum_{\bsh\in I} |\hat{f}(\bsh)-\left<f,\exp(2\pi i \bsh \cdot \bsx)  \right>_N|^2+1-\sum_{\bsh\in I} |\hat{f}(\bsh)|^2,
\end{align*}
where $\Vert f\Vert_{L_2}=1$ and $\left<f,\exp(2\pi i \bsh \cdot \bsx)  \right>_N$ is the discrete inner product of $\mathcal{Q}_If$ based on the rank-1 lattice points.  The  corresponding generating vector \( \boldsymbol{z} \) is constructed using Algorithm~\ref{Algorithm_1}. The weight sequence is chosen as \( \gamma_j = 1/j^{3.5},\;j=1,\ldots,d\), and the associated weighted hyperbolic cross index set is given by \( I = \mathcal{A}_{d,N^{3.4},3.4,\bsgamma} \). It is worth noting that we set the smoothness parameter to \( \alpha = 4 \) in Algorithm~\ref{Algorithm_2}, although \( \alpha \) is not necessarily an integer in the function space considered.   Moreover, to verify that QMC hyperinterpolation can overcome the curse of dimensionality, we examine the convergence rate of $\mathcal{Q}_I f$ for the weighted version of $f$, defined as  
\begin{align*}
f^{\omega}(\bsx) = \prod_{j=1}^d \left(1 + \omega_j f(x_j)\right).
\end{align*}
In this case, we set the weight sequence for Algorithm~\ref{Algorithm_1} as $\gamma_j = 0.1 \cdot j^{-4}$ for $j = 1, \ldots, d$, and choose the weights of $f^{\omega}(\bsx)$ as $\omega_j = j^{-8}$ for $j = 1, \ldots, d$.  The resulting convergence behavior is presented in Figure~\ref{convergence}.

\begin{figure}[t]
  \centering
  \includegraphics[width=0.8\linewidth]{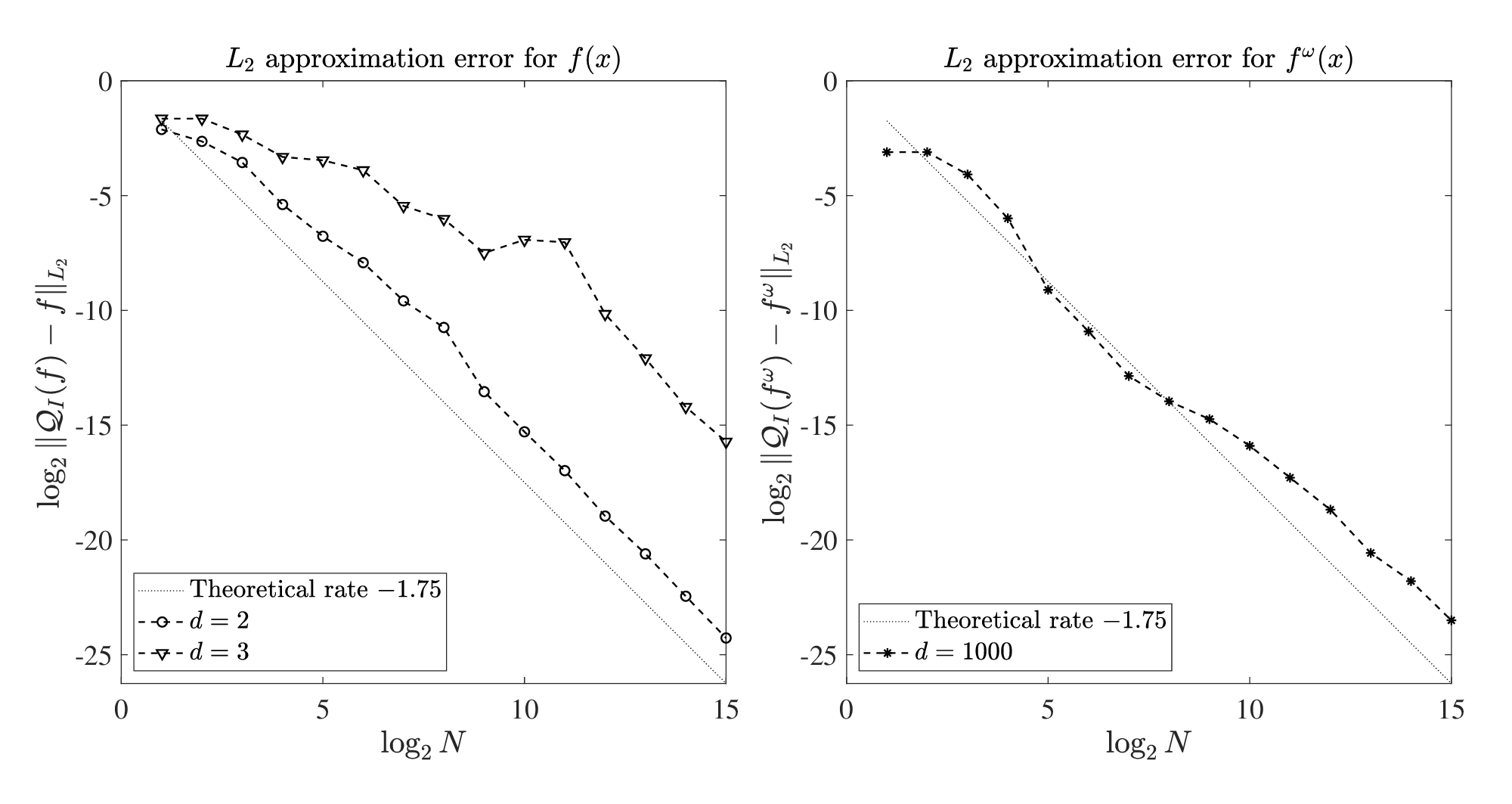}\\
\caption{The convergence behavior of QMC hyperinterpolation. } \label{convergence}
\end{figure}

Finally, we illustrate the denoising ability of the Lasso QMC hyperinterpolation $\mathcal{Q}_I^\lambda f$.  The level of noise is measured by the \emph{signal-to-noise ratio (SNR)}, which is defined as the ratio of signal to noisy data and is often expressed in decibels (dB).  A lower SNR scale suggests heavier noisy data, and we set all noise levels to $15$ dB. The parameter $\lambda$ is selected empirically; for strategies for choosing parameters in regularized problems, we refer the reader to \cite{Pereverzyev_2015_parameter}.

We present a two-dimensional denoising example for $f$, using the approximation scheme \( \mathcal{Q}_I^\lambda f \) based on the rank-1 lattice points. The corresponding results are shown in Figure~\ref{lattice_denoisy}. In addition, we include a one-dimensional example to test the denoising performance of \( \mathcal{Q}_I^{\lambda} f \) using polynomial lattice points.  Since the basis functions in \( \mathcal{Q}_I^{\lambda} f \) are one-dimensional Walsh functions, we choose two specific testing functions. The first is a square wave function with discontinuities precisely in dyadic rationals; the second is the Daubechies function generated by \texttt{cheb.gallery('daubechies')} in \textsc{CHEBFUN}~5.7.0~\cite{driscoll2014chebfun}. In this setting, for \( N = 2^m \), the quadrature points for \( \mathcal{Q}_I^{\lambda} f \) are the equidistant points on the interval \([0,1)\) with spacing \( 2^{-m} \), and the index set \( I \) is chosen as the integer interval \( \{0, 1, \dotsc, 2^{m-1} - 1\} \). The final recovery results are shown in Figure~\ref{Figure_one_denoisy}.

\begin{figure}[htbp]
  \centering
  \includegraphics[width=0.8\linewidth]{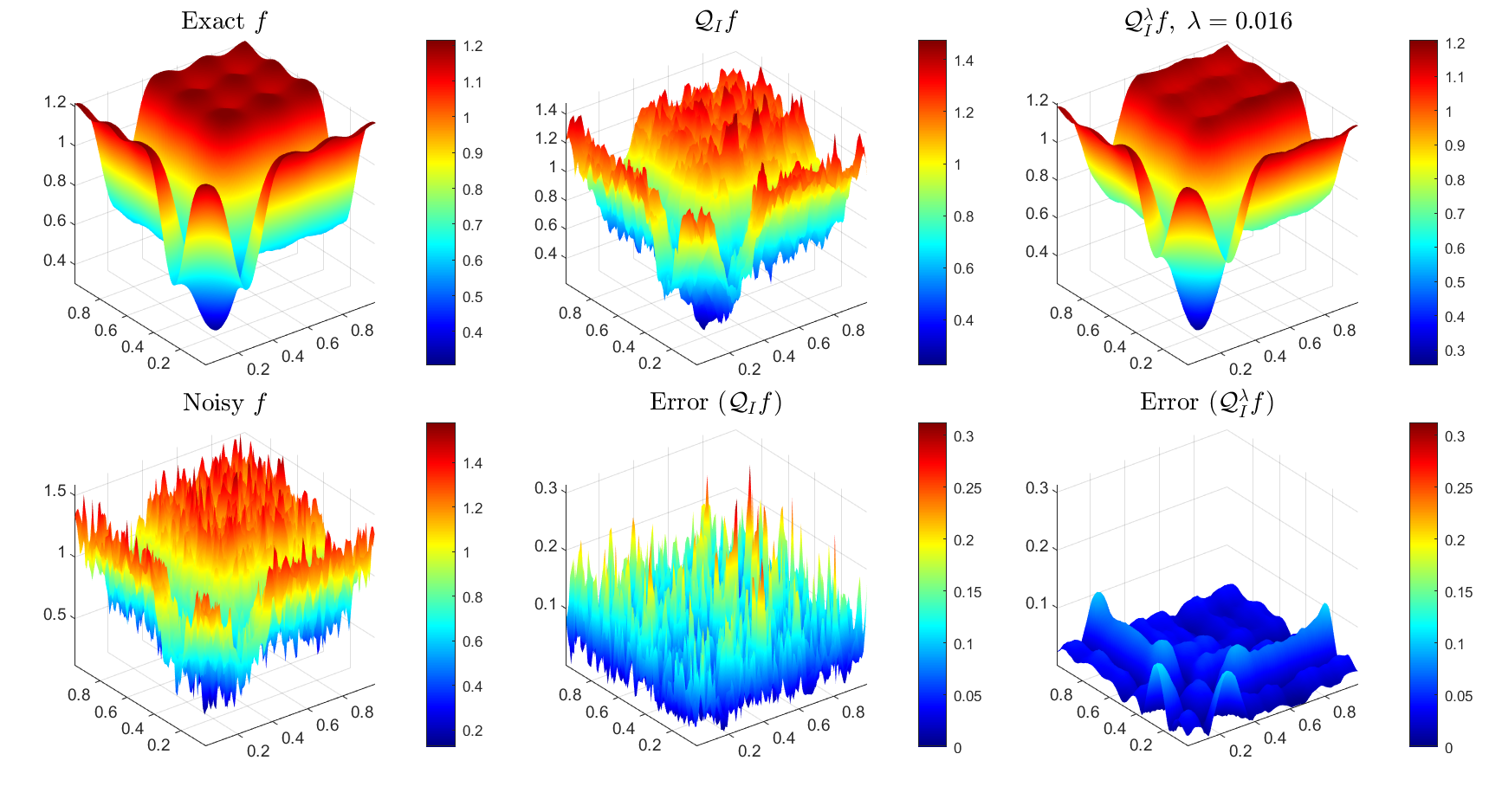}\\
\caption{Approximation results of $f(x)$ with $15$ dB noise via $\mathcal{Q}_I f$ and $\mathcal{Q}^\lambda_I f$ with $N=4093,\;\lambda=0.016$  } \label{lattice_denoisy}
\end{figure}

\begin{figure}[htbp]
  \centering
  \includegraphics[width=0.8\linewidth]{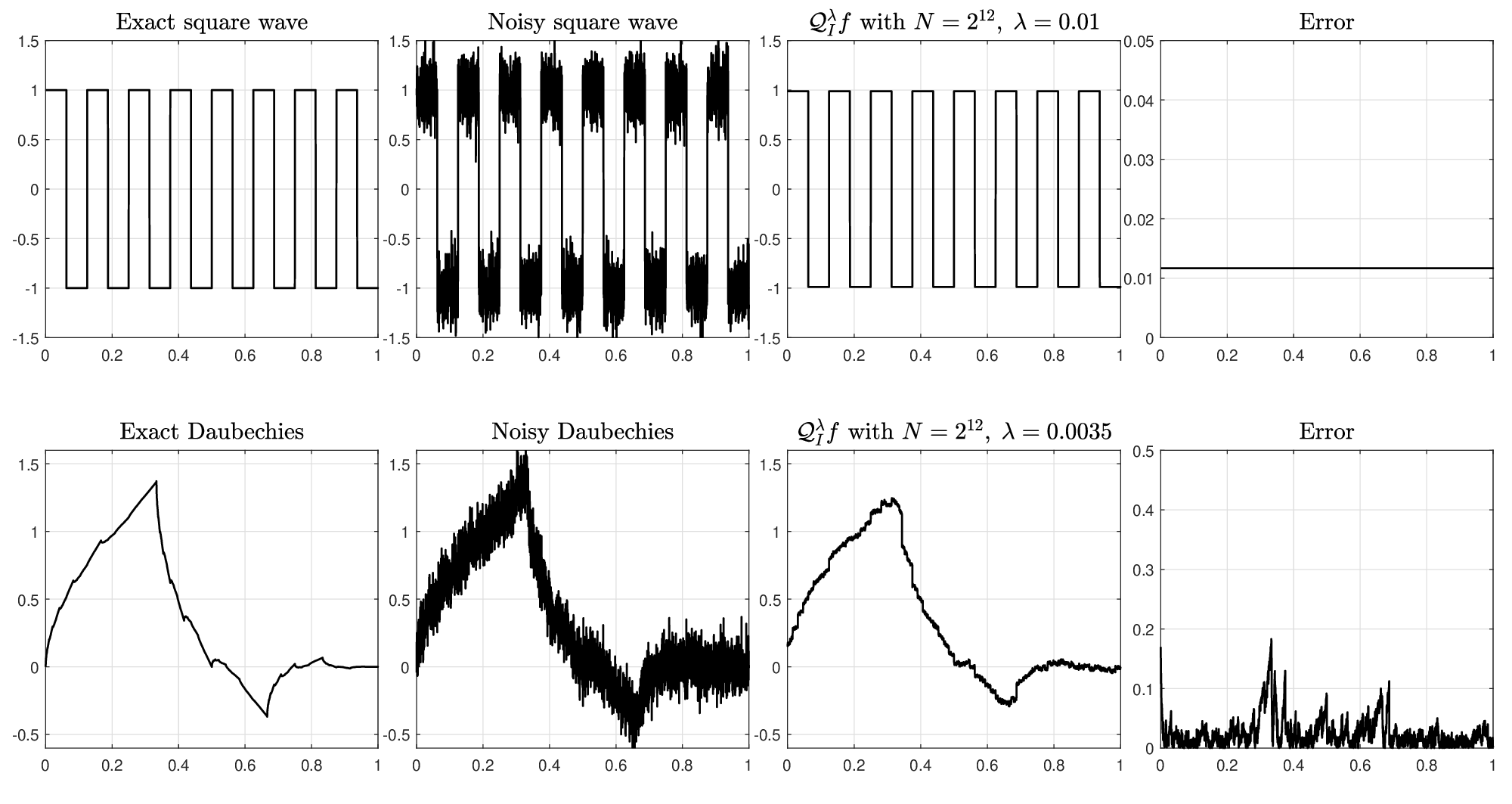}\\
\caption{Approximation results of square wave function and  Daubechies function with $15$ dB noise via $\mathcal{Q}_I f$ and $\mathcal{Q}^\lambda_I f$ with $N=2^{12}.$ } \label{Figure_one_denoisy}
\end{figure}

\appendix

\section{Walsh spaces are algebra}\label{app:Walsh_algebra}
Following \cite[Appendix~2]{Novak2004}, here we prove that the Walsh space $H_{d,\alpha,\bsgamma}^{\wal}$ is an algebra for any $\alpha>1/2$.
That is, for any $f,g\in H_{d,\alpha,\bsgamma}^{\wal}$, we have $fg\in H_{d,\alpha,\bsgamma}^{\wal}$ and
\[ \| fg\|_{d,\alpha,\bsgamma}^{\wal}\leq \breve{C}(d) \|f\|_{d,\alpha,\bsgamma}^{\wal} \|g\|_{d,\alpha,\bsgamma}^{\wal} \quad \text{with $\breve{C}(d)=2^d \prod_{j=1}^{d}\left(1+\gamma_{j}^2\frac{b-1}{b^{2\alpha}-b}\right)^{1/2}$.} \]

In what follows, we denote the digit-wise addition and subtraction modulo $b$ by $\oplus$ and $\ominus$, respectively. As preparation, we prove the following lemma.
\begin{lemma}\label{lem:mu1_property}
For any $h,k\in \NN_0$, it holds that
\[ \mu_1(h)\leq \max\left( \mu_1(k), \mu_1(h\ominus k)\right).\] 
\end{lemma}
\begin{proof}
    The case with $\mu_1(h)\leq \mu_1(k)$ is trivial. For the case with $\mu_1(h)> \mu_1(k)$, we have $\mu_1(h\ominus k)=\mu_1(h)$. Thus, the result holds with equality.
\end{proof}

It follows from \cite[Corollary~A.5]{Dick_Pillichshammer_2010} that
\begin{align*}
    f(\bsx)g(\bsx) = \sum_{\bsh,\bsk\in \NN_0^d}\breve{f}(\bsh)\breve{g}(\bsk)\wal_{\bsh\oplus \bsk}(\bsx) = \sum_{\bsk\in \NN_0^d}\left(\sum_{\bsh\in \NN_0^d}\breve{f}(\bsh)\breve{g}(\bsk\ominus \bsh)\right)\wal_{\bsk}(\bsx)
\end{align*}
Thus, we have
\[ \left(\| fg\|_{d,\alpha,\bsgamma}^{\wal}\right)^2 = \sum_{\bsk\in \NN_0^d}\breve{r}_{\alpha,\bsgamma}^2(\bsk)\left|\sum_{\bsh\in \NN_0^d}\breve{f}(\bsh)\breve{g}(\bsk\ominus \bsh)\right|^2. \]

By using Lemma~\ref{lem:mu1_property}, for any fixed $\bsk\in \NN_0^d$, we obtain
\begin{align*}
    \breve{r}_{\alpha,\bsgamma}(\bsk) & = \prod_{j=1}^{d} \max\left(1,\frac{b^{\alpha \mu_1(k_j)}}{\gamma_{j}}\right)\leq \prod_{j=1}^{d} \max\left(1,\frac{b^{\alpha \mu_1(h_j)}+b^{\alpha \mu_1(k_j\ominus h_j)}}{\gamma_{j}}\right)\\
    & = \sum_{u\subseteq \{1,\ldots,d\}} \prod_{j\in u}\max\left(1,\frac{b^{\alpha \mu_1(h_j)}}{\gamma_{j}}\right)\prod_{j'\in -u}\max\left(1,\frac{b^{\alpha \mu_1(k_j\ominus h_j)}}{\gamma_{j}}\right)\\
    & = \sum_{u\subseteq \{1,\ldots,d\}} \breve{r}_{\alpha,\bsgamma}(\bsh_u,\bszero)\breve{r}_{\alpha,\bsgamma}((\bsk\ominus \bsh)_{-u},\bszero),
\end{align*}
for all $\bsh\in \NN_0^d$. For $u\subseteq \{1,\ldots,d\}$, let us define
\begin{align*}
    F_u(\bsx) & = \sum_{\bsh\in \NN_0^d}\left|\breve{f}(\bsh)\right| \breve{r}_{\alpha,\bsgamma}(\bsh_u,\bszero)\wal_{\bsh}(\bsx), \\
    G_u(\bsx) & = \sum_{\bsh\in \NN_0^d}\left|\breve{g}(\bsh)\right| \breve{r}_{\alpha,\bsgamma}(\bsh_u,\bszero)\wal_{\bsh}(\bsx),
\end{align*}
both of which are well-defined functions in $L_2([0,1)^d)$, since $\breve{r}_{\alpha,\bsgamma}(\bsh_u,\bszero)\leq \breve{r}_{\alpha,\bsgamma}(\bsh)$ for all $\bsh$ and $u$, and moreover, $f$ and $g$ belong to $H_{d,\alpha,\bsgamma}^{\wal}$. Using these functions and the triangle inequality, for any $\bsk\in \NN_0^d$, we have
\begin{align*}
    \breve{r}_{\alpha,\bsgamma}(\bsk)\left|\sum_{\bsh\in \NN_0^d}\breve{f}(\bsh)\breve{g}(\bsk\ominus \bsh)\right| & \leq \breve{r}_{\alpha,\bsgamma}(\bsk)\sum_{\bsh\in \NN_0^d}|\breve{f}(\bsh)|\, |\breve{g}(\bsk\ominus \bsh)|\\
    & \leq \sum_{u\subseteq \{1,\ldots,d\}} \sum_{\bsh\in \NN_0^d}\breve{r}_{\alpha,\bsgamma}(\bsh_u,\bszero)\breve{r}_{\alpha,\bsgamma}((\bsk\ominus \bsh)_{-u},\bszero) \\
    &\qquad\times|\breve{f}(\bsh)|\, |\breve{g}(\bsk\ominus \bsh)|\\
    & = \sum_{u\subseteq \{1,\ldots,d\}} \sum_{\bsh\in \NN_0^d}\breve{F}_u(\bsh)\breve{G}_{-u}(\bsk\ominus \bsh).
\end{align*}
This leads to a bound
\begin{align*}
    \left(\| fg\|_{d,\alpha,\bsgamma}^{\wal}\right)^2 & \leq \sum_{\bsk\in \NN_0^d}\left( \sum_{u\subseteq \{1,\ldots,d\}} \sum_{\bsh\in \NN_0^d}\breve{F}_u(\bsh)\breve{G}_{-u}(\bsk\ominus \bsh)\right)^2 \\
    & \leq 2^d\sum_{u\subseteq \{1,\ldots,d\}}\sum_{\bsk\in \NN_0^d} \left( \sum_{\bsh\in \NN_0^d}\breve{F}_u(\bsh)\breve{G}_{-u}(\bsk\ominus \bsh)\right)^2,
\end{align*} 
where the last inequality follows from the Cauchy--Schwarz inequality.

For fixed $u$, we give a bound on 
\[ a_u = \sum_{\bsk\in \NN_0^d}\left( \sum_{\bsh\in \NN_0^d}\breve{F}_u(\bsh)\breve{G}_{-u}(\bsk\ominus \bsh)\right)^2. \]
We have
\begin{align*}
    a_u & = \sum_{\substack{\bsk_u\in \NN_0^{|u|}\\ \bsk_{-u}\in \NN_0^{d-|u|}}}\left( \sum_{\substack{\bsh_u\in \NN_0^{|u|}\\ \bsh_{-u}\in \NN_0^{d-|u|}}}\breve{F}_u(\bsh_u,\bsh_{-u})\breve{G}_{-u}(\bsk_u\ominus \bsh_u, \bsk_{-u}\ominus \bsh_{-u})\right)^2\\
    & = \sum_{\substack{\bsk_u\in \NN_0^{|u|}\\ \bsk_{-u}\in \NN_0^{d-|u|}}}\sum_{\substack{\bsh_u\in \NN_0^{|u|}\\ \bsh_{-u}\in \NN_0^{d-|u|}}}\sum_{\substack{\bsl_u\in \NN_0^{|u|}\\ \bsl_{-u}\in \NN_0^{d-|u|}}}\breve{F}_u(\bsk_u\ominus \bsh_u,\bsh_{-u})\breve{F}_u(\bsk_u\ominus \bsl_u,\bsl_{-u})\\
    & \qquad \qquad \qquad \times \breve{G}_{-u}(\bsh_u, \bsk_{-u}\ominus \bsh_{-u}) \breve{G}_{-u}(\bsl_u, \bsk_{-u}\ominus \bsl_{-u}).
\end{align*}
Here, applying the Cauchy--Schwarz inequality leads to
\begin{align*}
    & \sum_{\bsk_u\in \NN_0^{|u|}}\breve{F}_u(\bsk_u\ominus \bsh_u,\bsh_{-u})\breve{F}_u(\bsk_u\ominus \bsl_u,\bsl_{-u})\\
    & \leq \left(\sum_{\bsk_u\in \NN_0^{|u|}}(\breve{F}_u(\bsk_u\ominus \bsh_u,\bsh_{-u}))^2\right)^{1/2}\left(\sum_{\bsk_u\in \NN_0^{|u|}}(\breve{F}_u(\bsk_u\ominus \bsl_u,\bsl_{-u}))^2\right)^{1/2}\\
    & = \left(\sum_{\bsk_u\in \NN_0^{|u|}}(\breve{F}_u(\bsk_u,\bsh_{-u}))^2\right)^{1/2}\left(\sum_{\bsk_u\in \NN_0^{|u|}}(\breve{F}_u(\bsk_u,\bsl_{-u}))^2\right)^{1/2}.
\end{align*}
Similarly, we have
\begin{align*}
    & \sum_{\bsk_{-u}\in \NN_0^{d-|u|}}\breve{G}_{-u}(\bsh_u, \bsk_{-u}\ominus \bsh_{-u}) \breve{G}_{-u}(\bsl_u, \bsk_{-u}\ominus \bsl_{-u}) \\
    & \leq \left(\sum_{\bsk_{-u}\in \NN_0^{d-|u|}}(\breve{G}_{-u}(\bsh_u, \bsk_{-u}))^2\right)^{1/2}\left(\sum_{\bsk_{-u}\in \NN_0^{d-|u|}}(\breve{G}_{-u}(\bsl_u, \bsk_{-u}))^2\right)^{1/2}.
\end{align*}
With these bounds, we further have
\begin{align*}
    a_u & \le \left(\sum_{\bsh_{-u}\in \NN_0^{d-|u|}}\left(\sum_{\bsk_u\in \NN_0^{|u|}}(\breve{F}_u(\bsk_u,\bsh_{-u}))^2\right)^{1/2}\right)^{2}\\
    & \qquad\qquad\qquad\times \left(\sum_{\bsh_u\in \NN_0^{|u|}}\left(\sum_{\bsk_{-u}\in \NN_0^{d-|u|}}(\breve{G}_{-u}(\bsh_u, \bsk_{-u}))^2\right)^{1/2}\right)^2.
\end{align*}
Regarding the last two factors, it holds that
\begin{align*}
    & \left(\sum_{\bsh_{-u}\in \NN_0^{d-|u|}}\left(\sum_{\bsk_u\in \NN_0^{|u|}}(\breve{F}_u(\bsk_u,\bsh_{-u}))^2\right)^{1/2}\right)^{2} \\
    & = \left(\sum_{\bsh_{-u}\in \NN_0^{d-|u|}}\left(\sum_{\bsk_u\in \NN_0^{|u|}}\left|\breve{f}(\bsk_u,\bsh_{-u})\right|^2 (\breve{r}_{\alpha,\bsgamma}(\bsk_u,\bsh_{-u}))^2\right)^{1/2}\frac{1}{\breve{r}_{\alpha,\bsgamma}(\bsh_{-u},\bszero)}\right)^{2}\\
    & \leq \sum_{\bsh_{-u}\in \NN_0^{d-|u|}}\sum_{\bsk_u\in \NN_0^{|u|}}\left|\breve{f}(\bsk_u,\bsh_{-u})\right|^2 (\breve{r}_{\alpha,\bsgamma}(\bsk_u,\bsh_{-u}))^2\times   \sum_{\bsh_{-u}\in \NN_0^{d-|u|}}\frac{1}{(\breve{r}_{\alpha,\bsgamma}(\bsh_{-u},\bszero))^2}\\
    & = \left(\| f\|_{d,\alpha,\bsgamma}^{\wal}\right)^2 \sum_{\bsh_{-u}\in \NN_0^{d-|u|}}\frac{1}{(\breve{r}_{\alpha,\bsgamma}(\bsh_{-u},\bszero))^2},
\end{align*}
as well as
\begin{align*}
    \left(\sum_{\bsh_u\in \NN_0^{|u|}}\left(\sum_{\bsk_{-u}\in \NN_0^{d-|u|}}(\breve{G}_{-u}(\bsh_u, \bsk_{-u}))^2\right)^{1/2}\right)^2\leq \left(\| g\|_{d,\alpha,\bsgamma}^{\wal}\right)^2  \sum_{\bsh_u\in \NN_0^{|u|}}\frac{1}{(\breve{r}_{\alpha,\bsgamma}(\bsh_u,\bszero))^2}.
\end{align*}

Altogether we obtain
\begin{align*}
    \left(\| fg\|_{d,\alpha,\bsgamma}^{\wal}\right)^2 
    & \leq 2^d\sum_{u\subseteq \{1,\ldots,d\}}a_u\\
    & \leq 2^d\left(\| f\|_{d,\alpha,\bsgamma}^{\wal}\right)^2\left(\| g\|_{d,\alpha,\bsgamma}^{\wal}\right)^2\\
    &\qquad\qquad\times\sum_{u\subseteq \{1,\ldots,d\}}\sum_{\bsh_{-u}\in \NN_0^{d-|u|}}\frac{1}{(\breve{r}_{\alpha,\bsgamma}(\bsh_{-u},\bszero))^2}\sum_{\bsh_u\in \NN_0^{|u|}}\frac{1}{(\breve{r}_{\alpha,\bsgamma}(\bsh_u,\bszero))^2}\\
    & = 4^d\left(\| f\|_{d,\alpha,\bsgamma}^{\wal}\right)^2\left(\| g\|_{d,\alpha,\bsgamma}^{\wal}\right)^2\sum_{\bsh\in \NN_0^{d}}\frac{1}{(\breve{r}_{\alpha,\bsgamma}(\bsh))^2}.
\end{align*}
Finally, the last sum over $\bsh$ is given by
\begin{align*}
    \sum_{\bsh\in \NN_0^{d}}\frac{1}{(\breve{r}_{\alpha,\bsgamma}(\bsh))^2} & = \prod_{j=1}^{d}\sum_{h_j=0}^{\infty}\min\left(1,\frac{\gamma^2_{j}}{b^{2\alpha \mu_1(h_j)}}\right) = \prod_{j=1}^{d}\left(1+\gamma_{j}^2\sum_{\ell_j=1}^{\infty}\sum_{h_j=b^{\ell_j-1}}^{b^{\ell_j}-1}\frac{1}{b^{2\alpha \mu_1(h_j)}}\right)\\
    & = \prod_{j=1}^{d}\left(1+\gamma_{j}^2\frac{b-1}{b^{2\alpha}-b}\right).
\end{align*}
This concludes the result.

\section*{Funding} 
The ﬁrst author (C. An) is partially supported by National Natural Science Foundation of China (No. 12371099), and the Special Posts of Guizhou University (No. [2024]42).
The third author (T. Goda) is supported by JSPS KAKENHI Grant Number 23K03210.

\end{document}